\documentclass[a4paper]{amsart}

\usepackage{framed}
\usepackage{xargs}
\usepackage[pdftex,dvipsnames]{xcolor}
\usepackage[colorinlistoftodos,prependcaption,textsize=tiny]{todonotes}
\newcommandx{\commentmark}[2][1=]{\todo[linecolor=Blue,backgroundcolor=Blue!25,bordercolor=Blue,#1]{M: #2}}
\newcommandx{\commentrob}[2][1=]{\todo[linecolor=Orange,backgroundcolor=Orange!25,bordercolor=Orange,#1]{R: #2}}
\newcommandx{\commentvalentina}[2][1=]{\todo[linecolor=Green,backgroundcolor=Green!25,bordercolor=Green,#1]{V: #2}}

\usepackage{longtable}

\usepackage[english]{babel}
\usepackage{amssymb, amsmath, amscd, amsfonts, amsthm}
\usepackage{subcaption}
\usepackage{subfloat}
\usepackage{url}

\usepackage{bbm}
\newcommand{\one}{\mathbbm{1}}

\usepackage[outline]{contour}
\contourlength{0.15em}

\theoremstyle{plain}
\numberwithin{equation}{section}
\newtheorem{theorem}[equation]{Theorem}
\newtheorem*{theorem*}{Theorem}
\newtheorem{corollary}[equation]{Corollary}
\newtheorem*{corollary*}{Corollary}
\newtheorem{lemma}[equation]{Lemma}
\newtheorem*{lemma*}{Lemma}
\newtheorem{proposition}[equation]{Proposition}
\newtheorem*{proposition*}{Proposition}

\theoremstyle{definition}
\newtheorem{definition}[equation]{Definition}
\newtheorem{remark}[equation]{Remark}
\newtheorem{notation}[equation]{Notation}
\newtheorem*{acknowledgements}{Acknowledgements}

\newcommand{\fakeenv}{} 

\newenvironment{restate}[2]  
{
	\renewcommand{\fakeenv}{#2} 
	\theoremstyle{plain}
	\newtheorem*{\fakeenv}{#1~\ref{#2}} 
	\begin{\fakeenv}
}
{
	\end{\fakeenv}
}

\newcommand{\DeclareMathBinary}[2]{\newcommand{#1}{\mathbin{#2}}}
\newcommand{\DeclareMathRelation}[2]{\newcommand{#1}{\mathrel{#2}}}

\newcommand{\ZZ}{\mathbb{Z}}
\newcommand{\calA}{\mathcal{A}}

\newcommand{\calR}{\mathcal{R}}
\newcommand{\calT}{\mathcal{T}}
\DeclareMathOperator{\Sym}{Sym}
\DeclareMathOperator{\Cr}{Cr}
\DeclareMathOperator{\lnk}{link}

\DeclareMathBinary{\crosses}{\pitchfork}
\DeclareMathOperator{\arc}{arc}
\DeclareMathOperator{\pol}{\mathcal{P}}
\DeclareMathOperator{\Hyp}{\mathcal{H}}
\DeclareMathOperator{\flip}{\mathcal{F}}
\DeclareMathOperator{\EMod}{Mod^{\pm}} 
\DeclareMathOperator{\Aut}{Aut} 

\DeclareMathOperator{\intersection}{\iota}

\DeclareMathRelation{\isom}{\cong} 
\DeclareMathRelation{\homeo}{\cong} 

\newcommand{\from}{\colon} 

\usepackage{colonequals}
\newcommand{\defeq}{\colonequals}

\usepackage{tikz}
\usetikzlibrary{calc}
\usetikzlibrary{patterns}
\usetikzlibrary{intersections}
\usetikzlibrary{decorations.pathreplacing}
\usetikzlibrary{decorations.pathmorphing}
\usetikzlibrary{decorations.markings}
\usetikzlibrary{shapes.geometric}

\tikzset{dot/.style={draw,shape=circle,fill=black,scale=0.4}}
\tikzset{tri/.style={draw,scale=0.2,fill=white,regular polygon,regular polygon sides=3}}

\makeatletter
\DeclareRobustCommand{\rvdots}{%
  \vbox{
    \baselineskip4\p@\lineskiplimit\z@
    \kern-\p@
    \hbox{.}\hbox{.}\hbox{.}
  }}
\makeatother

\usepackage[hidelinks, pagebackref, pdftex]{hyperref}

\renewcommand*{\backref}[1]{}
\renewcommand*{\backrefalt}[4]{
	\ifcase #1 %
		[No citations.]%
	\else
		[#2]%
	\fi
}

\title{Cubical geometry in the polygonalisation complex}

\author{Mark C. Bell}
\address{Department of Mathematics, University of Illinois, Urbana, IL}
\email{\url{mcbell@illinois.edu}}

\author{Valentina Disarlo}
\address{Mathematical Sciences Research Institute, Berkeley, CA}
\email{\url{v.disarlo@gmail.com}}

\author{Robert Tang}
\address{Department of Mathematics, University of Oklahoma, Norman, OK}
\email{\url{rtang@math.ou.edu}}

\begin{document}

\begin{abstract}
We introduce the polygonalisation complex of a surface, a cube complex whose vertices correspond to polygonalisations.
This is a geometric model for the mapping class group and it is motivated by works of Harer, Mosher and Penner.
Using properties of the flip graph, we show that the midcubes in the polygonalisation complex can be extended to a family of embedded and separating hyperplanes, parametrised by the arcs in the surface.

We study the crossing graph of these hyperplanes and prove that it is quasi-isometric to the arc complex.
We use the crossing graph to prove that, generically, different surfaces have different polygonalisation complexes. 
The polygonalisation complex is not CAT(0), but we can characterise the vertices where Gromov's link condition fails.
This gives a tool for proving that, generically, the automorphism group of the polygonalisation complex is the (extended) mapping class group of the surface.
\end{abstract}

\maketitle

\begin{figure}[ht]
\centering
\begin{tikzpicture}[scale=4.5, thick]  

\pgfmathsetmacro{\inner}{0.5}
\coordinate (N) at (0,1);
\coordinate (S) at (0,-1);
\coordinate (A) at (-0.8,0.2);
\coordinate (B) at (0.2,0.0);
\coordinate (C) at (1,0.2);
\coordinate (X) at (0.1,0.1);

\coordinate (d1A) at (-0.2,-0.2);
\coordinate (d2A) at (0.2,-0.2);
\coordinate (d1B) at (-0.2,-0.2);
\coordinate (d2B) at (0.2,-0.2);
\coordinate (d1C) at (-0.2,-0.2);
\coordinate (d2C) at (0.2,-0.2);

\coordinate (Ac) at ($(A)+0.5*(d1A)+0.5*(d2A)$);
\coordinate (Bc) at ($(B)+0.5*(d1B)+0.5*(d2B)$);
\coordinate (Cc) at ($(C)+0.5*(d1C)+0.5*(d2C)$);

\coordinate (ABt) at ($0.33*(A)+0.33*(B)+0.33*(N)$);
\coordinate (ACt) at ($0.33*(A)+0.33*(C)+0.23*(N)$);
\coordinate (CBt) at ($0.33*(C)+0.33*(B)+0.33*(N)$);

\coordinate (A2) at ($(A)+(d1A)+(d2A)$);
\coordinate (B2) at ($(B)+(d1B)+(d2B)$);
\coordinate (C2) at ($(C)+(d1C)+(d2C)$);

\coordinate (A2B2b) at ($0.33*(A2)+0.33*(B2)+0.33*(S)$);
\coordinate (A2C2b) at ($0.33*(A2)+0.33*(C2)+0.23*(S)$);
\coordinate (C2B2b) at ($0.33*(C2)+0.33*(B2)+0.33*(S)$);

\fill [blue, opacity=0.2] (N) -- ($(N)!0.5!(C)$) -- (CBt) -- (X) -- (ABt) -- ($(N)!0.5!(A)$) -- cycle;
\fill [gray, opacity=0.2] ($(N)!0.5!(B)$) -- (CBt) -- ($(B)+0.5*(d2B)$) -- ($(B)+0.5*(d1B)+0.5*(d2B)$) -- ($(B)+0.5*(d1B)$) -- (ABt) -- cycle;

\draw [very thick] (N) -- (A);
\draw [very thick] (A) -- ($(A)+(d1A)$);
\draw [dotted] (A) -- ($(A)+(d2A)$);
\draw [very thick] ($(A)+(d1A)$) -- ($(A)+(d1A)+(d2A)$);
\draw [dotted] ($(A)+(d2A)$) -- ($(A)+(d1A)+(d2A)$);
\draw [very thick] ($(A)+(d1A)+(d2A)$) -- (S);

\draw [very thick] (N) -- (B);
\draw [very thick] (B) -- ($(B)+(d1B)$);
\draw [very thick] (B) -- ($(B)+(d2B)$);
\draw [very thick] ($(B)+(d1B)$) -- ($(B)+(d1B)+(d2B)$);
\draw [very thick] ($(B)+(d2B)$) -- ($(B)+(d1B)+(d2B)$);
\draw [very thick] ($(B)+(d1B)+(d2B)$) -- (S);

\draw [very thick] (N) -- (C);
\draw [dotted] (C) -- ($(C)+(d1C)$);
\draw [very thick] (C) -- ($(C)+(d2C)$);
\draw [dotted] ($(C)+(d1C)$) -- ($(C)+(d1C)+(d2C)$);
\draw [very thick] ($(C)+(d2C)$) -- ($(C)+(d1C)+(d2C)$);
\draw [very thick] ($(C)+(d1C)+(d2C)$) -- (S);

\draw [thin, dotted, gray] ($(A)+(d2A)$) -- ($(C)+(d1C)$);
\draw [very thick] ($(A)+(d1A)$) -- ($(B)+(d1B)$);
\draw [very thick] ($(B)+(d2B)$) -- ($(C)+(d2C)$);

\draw [dotted] ($(A)+0.5*(d1A)$) -- (Ac);
\draw [dotted] ($(A)+0.5*(d2A)$) -- (Ac);
\draw [dotted] ($(A)+(d1A)+0.5*(d2A)$) -- (Ac);
\draw [dotted] ($(A)+0.5*(d1A)+(d2A)$) -- (Ac);

\draw ($(B)+0.5*(d1B)$) -- (Bc);
\draw ($(B)+0.5*(d2B)$) -- (Bc);
\draw ($(B)+(d1B)+0.5*(d2B)$) -- (Bc);
\draw ($(B)+0.5*(d1B)+(d2B)$) -- (Bc);

\draw [dotted] ($(C)+0.5*(d1C)$) -- (Cc);
\draw [dotted] ($(C)+0.5*(d2C)$) -- (Cc);
\draw [dotted] ($(C)+(d1C)+0.5*(d2C)$) -- (Cc);
\draw [dotted] ($(C)+0.5*(d1C)+(d2C)$) -- (Cc);

\draw ($(N)!0.5!(A)$) -- (ABt);
\draw ($(N)!0.5!(B)$) -- (ABt);
\draw ($(A)+0.5*(d1A)$) -- (ABt);
\draw ($(B)+0.5*(d1B)$) -- (ABt);
\draw ($0.5*(A)+0.5*(d1A)+0.5*(B)+0.5*(d1B)$) -- (ABt);

\draw [thin, dotted, gray] ($(N)!0.5!(A)$) -- (ACt);
\draw [thin, dotted, gray] ($(N)!0.5!(C)$) -- (ACt);
\draw [thin, dotted, gray] ($(A)+0.5*(d2A)$) -- (ACt);
\draw [thin, dotted, gray] ($(C)+0.5*(d1C)$) -- (ACt);
\draw [thin, dotted, gray] ($0.5*(A)+0.5*(d1A)+0.5*(C)+0.5*(d1C)$) -- (ACt);

\draw ($(N)!0.5!(C)$) -- (CBt);
\draw ($(N)!0.5!(B)$) -- (CBt);
\draw ($(C)+0.5*(d2C)$) -- (CBt);
\draw ($(B)+0.5*(d2B)$) -- (CBt);
\draw ($0.5*(C)+0.5*(d2C)+0.5*(B)+0.5*(d2B)$) -- (CBt);

\draw ($(S)!0.5!(A2)$) -- (A2B2b);
\draw ($(S)!0.5!(B2)$) -- (A2B2b);
\draw ($(A2)-0.5*(d2A)$) -- (A2B2b);
\draw ($(B2)-0.5*(d2B)$) -- (A2B2b);
\draw ($0.5*(A2)-0.5*(d2A)+0.5*(B2)-0.5*(d2B)$) -- (A2B2b);

\draw [thin, dotted, gray] ($(S)!0.5!(A2)$) -- (A2C2b);
\draw [thin, dotted, gray] ($(S)!0.5!(C2)$) -- (A2C2b);
\draw [thin, dotted, gray] ($(A2)-0.5*(d1A)$) -- (A2C2b);
\draw [thin, dotted, gray] ($(C2)-0.5*(d2C)$) -- (A2C2b);
\draw [thin, dotted, gray] ($0.5*(A2)-0.5*(d2A)+0.5*(C2)-0.5*(d2C)$) -- (A2C2b);

\draw ($(S)!0.5!(C2)$) -- (C2B2b);
\draw ($(S)!0.5!(B2)$) -- (C2B2b);
\draw ($(C2)-0.5*(d1C)$) -- (C2B2b);
\draw ($(B2)-0.5*(d1B)$) -- (C2B2b);
\draw ($0.5*(C2)-0.5*(d1C)+0.5*(B2)-0.5*(d1B)$) -- (C2B2b);

\draw [dotted] (ABt) -- (X);
\draw [dotted] (CBt) -- (X);
\draw [thin, dotted, gray] (ACt) -- (X);
\draw [dotted] (X) -- ($(A)+0.5*(d1A)+0.5*(d2A)$);
\draw [dotted] (X) -- ($(B)+0.5*(d1B)+0.5*(d2B)$);
\draw [dotted] (X) -- ($(C)+0.5*(d1C)+0.5*(d2C)$);
\draw [dotted] (A2B2b) -- (X);
\draw [dotted] (C2B2b) -- (X);
\draw [thin, dotted, gray] (A2C2b) -- (X);

\foreach \p in {N,S,A,B,C,A2,B2,C2} {\node [dot, red] at (\p) {};}
\foreach \p in {($(A)+(d1A)$),($(B)+(d1B)$),($(B)+(d2B)$),($(C)+(d2C)$)} {\node [dot, red] at \p {};}
\foreach \p in {($(A)+(d2A)$),($(C)+(d1C)$)} {\node [dot, red!50] at \p {};}  

\foreach \p in {A,B,C} {\node [dot,blue] at ($(N)!0.5!(\p)$) {};}
\foreach \p in {A2,B2,C2} {\node [dot,blue] at ($(S)!0.5!(\p)$) {};}
\foreach \p in {($(A)+0.5*(d1A)$),($(A2)-0.5*(d2A)$),($(B)+0.5*(d1B)$),($(B2)-0.5*(d2B)$),($(B)+0.5*(d2B)$),($(B2)-0.5*(d1B)$),($(C)+0.5*(d2C)$),($(C2)-0.5*(d1C)$),
	($0.5*(A)+0.5*(d1A)+0.5*(B)+0.5*(d1B)$), ($0.5*(B)+0.5*(d2B)+0.5*(C)+0.5*(d2C)$)} {\node [dot, blue] at \p {};}
\foreach \p in {($(A)+0.5*(d2A)$),($(A2)-0.5*(d1A)$),($(C)+0.5*(d1C)$),($(C2)-0.5*(d2C)$),($0.5*(A)+0.5*(d2A)+0.5*(C)+0.5*(d1C)$)} {\node [dot, blue!50] at \p {};}  

\foreach \p in {(ABt),(CBt),(A2B2b),(C2B2b),($(B)+0.5*(d1B)+0.5*(d2B)$)} {\node [dot, green] at \p {};}
\foreach \p in {(ACt),(A2C2b),($(A)+0.5*(d1A)+0.5*(d2A)$),($(C)+0.5*(d1C)+0.5*(d2C)$)} {\node [dot, green!50] at \p {};}  

\node [dot, orange] at (X) {};

\end{tikzpicture}
\caption
[The polygonalisation complex of a hexagon.]
{The polygonalisation complex of a hexagon $S_{0,0}^6 =$
\tikz[scale=0.3,line cap=round,baseline=-3]{
	\foreach \i in {0,60,...,360} {\draw (\i:1) -- (\i+60:1);}
	}.}
\label{fig:P_S_0_0^6}
\end{figure}
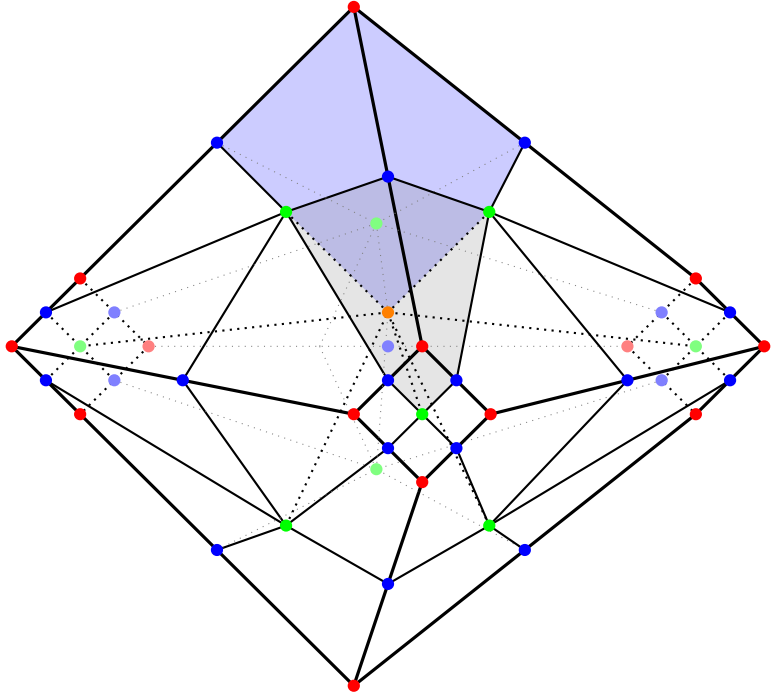

\clearpage 

\section{Introduction}

In this paper we study the polygonalisation complex $\pol(S)$ of a surface $S$ with marked points.
This is a cube complex encoding the combinatorics of the polygonalisations of $S$, that is, the multiarcs that decompose $S$ into polygons.
The polygonalisation complex contains (the barycentric subdivision of) the flip graph $\flip(S)$ as a subcomplex.
It can also be regarded as the barycentric subdivision of the contractible CW-complex naturally associated to the Ptolemy groupoid of $S$ which appears in quantum Teichm\"{u}ller theory \cite{Funar} \cite{Roger}.

The mapping class group $\EMod(S)$ acts geometrically on $\pol(S)$.
However, in general, $\EMod(S)$ is not a CAT(0) group \cite[Theorem~4.2]{KapovichLeeb} and so $\pol(S)$ is not a CAT(0) cube complex.
We show that, in spite of this, $\pol(S)$ has many of the properties of CAT(0) cube complexes.
In particular, it has a rich hyperplane structure that is closely related to the arc complex.

\begin{restate}{Theorem}{thrm:pol_sageev}
There is a natural one-to-one correspondence between the hyperplanes of $\pol(S)$ and the arcs on $S$.
Moreover, for each arc $\alpha \in \calA(S)$, the corresponding hyperplane $H_\alpha$ is embedded, two-sided and separates $\pol(S)$ into two connected components: $\pol_\alpha(S)$ and $\overline{\pol_\alpha}(S)$.
\end{restate}

One can encode the combinatorics of the hyperplanes in $\pol(S)$ via its \emph{crossing graph} $\Cr(\pol(S))$: this has a vertex for every hyperplane in $\pol(S)$, and two hyperplanes are connected via an edge if and only if they cross.

\begin{restate}{Proposition}{prop:equivalent_crossing}
Let $H_\alpha$ and $H_\beta$ be the hyperplanes of $\pol(S)$ corresponding to arcs $\alpha, \beta \in \calA(S)$.
Then $H_\alpha$ and $H_\beta$ cross if and only if $\alpha$ and $\beta$ are disjoint and do not bound a (folded) triangle.
\end{restate}

Together, Theorem~\ref{thrm:pol_sageev} and Proposition~\ref{prop:equivalent_crossing} show that there is a natural embedding of $\Cr(\pol(S))$ into the arc graph $\calA(S)$.
This embedding is in fact a quasi-isometry (Corollary~\ref{cor:qi}). 
In particular, when $\partial S = \emptyset$ the arc graph is Gromov hyperbolic \cite[Theorem~20.2]{MasurSchleimer}, hence the crossing graph is also.
Proposition~\ref{prop:equivalent_crossing} also enables us to characterise the edges of the arc graph that do not appear in the crossing graph (Lemma~\ref{lem:characterise_folded}).
We show that $\calA(S)$ can be recovered from the combinatorics of $\Cr(\pol(S))$.
Applying rigidity results for arc graphs \cite{IrmakMcCarthy} \cite{Disarlo} \cite{KorkmazPapadopoulos} we then obtain:

\begin{restate}{Theorem}{thrm:pol_isom}
For all but finitely many pairs of surfaces, the complexes $\pol(S)$ and $\pol(S')$ are isomorphic if and only if $S$ and $S'$ are homeomorphic.
\end{restate}

We rephrase Gromov's link condition in terms of curves on $S$ (Definition~\ref{def:pos_curvature}).
We use this to extract the number of components of a given polygonalisation from the local combinatorics of $\pol(S)$.

\begin{restate}{Corollary}{cor:deficiency_combinatorics}
For each $k$, there is a combinatorial criterion that characterises the vertices of $\pol(S)$ corresponding to polygonalisations with exactly $k$ arcs.
\end{restate}
This gives a method for proving the rigidity of the polygonalisation complex via flip graph rigidity.
\begin{restate}{Theorem}{thrm:pol_aut}
For all but finitely many surfaces, the natural homomorphism
\[ \EMod(S) \to \Aut(\pol(S)) \]
is an isomorphism.
\end{restate}

We list the exceptions to these theorems in Appendix~\ref{sec:exceptions}.
Additionally, we highlight these results do not follow from \cite[Theorem~1.1]{Aramayona} since $\pol(S)$ does not satisfy the required rigidity axioms.


\section{Preliminaries}

Let $S$ be a connected, compact, orientable surface with a finite non-empty set of marked points.
We assume that each boundary component (if any) contains at least one marked point.
When $\partial S \neq \emptyset$, such surfaces are also known in the literature as ciliated \cite[Section~2]{FockGoncharov}.
Let
\[ E(S) \defeq 6g + 3b + 3s + p - 6 \; \textrm{and} \; F(S) \defeq 4g + 2b + 2s + p - 4 \]
where $g$ is the genus of $S$, $s$ is the number of marked points in its interior, $b$ is the number of boundary components and $p$ is the number of marked points on $\partial S$.

To avoid pathologies, from now on we will require that $F(S) \geq 3$.
The \emph{exceptional surfaces}, for which $F(S) < 3$, are listed and discussed in Appendix~\ref{sec:exceptions}.

\subsection{Objects}

We recall some standard objects that will appear throughout:

\subsubsection{Mapping class group}
The \emph{(extended) mapping class group} $\EMod(S)$ is the group of homeomorphisms of $S$ relative to the marked points up to isotopy.
We allow mapping classes to reverse orientation, permute the marked points and permute the boundary components of $S$.

\subsubsection{Triangulations}
An \emph{(essential) arc} $\alpha$ on $S$ is an embedded arc connecting marked points up to isotopy (relative to the set of marked points).
Such arcs are not \emph{null-homotopic} and are not \emph{boundary-parallel}, that is, they do not cut off a monogon (Figure~\ref{fig:inessential_monogon}) or cut off a bigon together with part of $\partial S$ (Figure~\ref{fig:inessential_bigon}) respectively.
We write $\intersection(\alpha, \beta)$ for the \emph{(geometric) intersection number} of $\alpha$ and $\beta$.
Arcs $\alpha$ and $\beta$ have disjoint interiors, which we refer to simply as being \emph{disjoint}, if and only $\intersection(\alpha, \beta) = 0$.
A \emph{multiarc} is a set of distinct and pairwise disjoint arcs.
For example, see Figure~\ref{fig:multiarc}.

\begin{figure}[ht]
	\centering
	\begin{subfigure}[b]{0.5\textwidth}
		\centering
		\begin{tikzpicture}[scale=0.8, thick]

\fill [pattern=north east lines,opacity=0.6] (-1,0) to [out=60,in=90] (1,0) to [out=270,in=-60] (-1,0);

\draw [red] (-1,0) to [out=60,in=90] (1,0) to [out=270,in=-60] (-1,0);
\node [dot] at (-1,0) {};

\path [use as bounding box] (-1.25,0.75) rectangle (1.25,-0.75);

\end{tikzpicture}
		\caption{A monogon arc.}
		\label{fig:inessential_monogon}
	\end{subfigure}%
	~
	\begin{subfigure}[b]{0.5\textwidth}
		\centering
		\begin{tikzpicture}[scale=1, thick]

\fill [pattern=north east lines,opacity=0.6] (-1,0) to [out=60,in=120] (1,0) to (-1,0);

\draw [very thick] (-1.25,0) -- (1.25,0) node [right] {$\partial S$};
\draw [red] (-1,0) to [out=60,in=120] (1,0);
\node [dot] at (-1,0) {};
\node [dot] at (1,0) {};

\path [use as bounding box] (-1.25,0.75) rectangle (1.25,-0.75);

\end{tikzpicture}
		\caption{A bigon arc.}
		\label{fig:inessential_bigon}
	\end{subfigure}
	\caption{Inessential arcs in $S$.}
\end{figure}
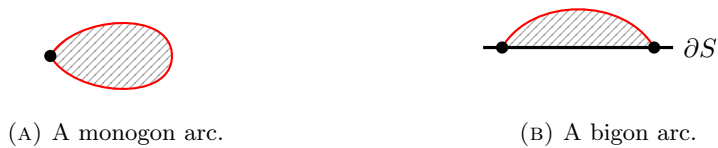

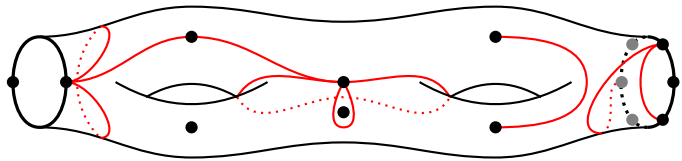
\begin{figure}[ht]
\centering
\begin{tikzpicture}[scale=2, thick]

\draw [red, dotted] (0.4,0.37) to [out=180+45,in=180-45] (0.4,-0.37);
\draw [red, dotted] (1.3,-0.1) to [out=-60,in=180] (2,-0.1) to [out=0,in=240] (2.7,-0.1);
\draw [red, dotted] (3.7,-0.34) to [out=30,in=180] (3.83,0);

\draw [dotted, very thick] (4,-0.3) to [out=180,in=180]
	(4,0.3);

\draw
	(0,0.3) to [out=0,in=180]
	(1,0.5) to [out=0,in=180]
	(2,0.5-0.1) to [out=0,in=180]
	(2,0.5-0.1) to [out=0,in=180]
	(3,0.5) to [out=0,in=180]
	(4,0.3);
\draw (4,-0.3) to [out=180,in=0]
	(3,-0.5) to [out=180,in=0]
	(2,-0.5+0.1) to [out=180,in=0]
	(1,-0.5) to [out=180,in=0]
	(0,-0.3);

\draw [very thick] (0,0.3) to [out=0,in=0]
	(0,-0.3) to [out=180,in=180]
	(0,0.3);

\draw [very thick] (4,0.3) to [out=0,in=0]
	(4,-0.3) ;

\draw [red] (0.175,0) to [out=0,in=0] (0.4,0.37);
\draw [red] (0.175,0) to [out=0,in=0] (0.4,-0.37);

\draw [red] (0.175,0) to [out=0,in=180] (1,0.3);

\draw [red] (1,0.3) to [out=0,in=180] (2,0);
\draw [red] (2,0) to [out=180,in=60] (1.3,-0.1);
\draw [red] (2,0) to [out=0,in=120] (2.7,-0.1);

\draw [red] (2,0) to [out=240,in=180] (2,-0.3) to [out=0,in=300] (2,0);

\draw [red] (3,0.3) to [out=0,in=90] (3.6,0) to [out=270,in=0] (3,-0.3);

\draw [red] (4.1,0.25) to [out=180,in=180] (4.1,-0.25);
\draw [red] (4.1,0.25) to [out=180,in=180] (3.7,-0.34);

\draw (0.5,0) to [out=-30,in=180+30] (1.5,0);
\draw (0.7,-0.1) to [out=30,in=180-30] (1.3,-0.1);
\draw (2.5,0) to [out=-30,in=180+30] (3.5,0);
\draw (2.7,-0.1) to [out=30,in=180-30] (3.3,-0.1);

\node [dot] at (0.175,0) {};
\node [dot] at (-0.175,0) {};
\node [dot] at (1,0.3) {};
\node [dot] at (1,-0.3) {};
\node [dot] at (2,0) {};
\node [dot] at (2,-0.2) {};
\node [dot] at (3,0.3) {};
\node [dot] at (3,-0.3) {};

\node [dot] at (4.1,0.25) {};
\node [dot] at (4.17,0) {};
\node [dot] at (4.1,-0.25) {};

\node [dot, gray] at (3.9,0.25) {};
\node [dot, gray] at (3.83,0) {};
\node [dot, gray] at (3.9,-0.25) {};

\end{tikzpicture}
\caption{A multiarc on $S$.}
\label{fig:multiarc}
\end{figure}

The set of multiarcs is a poset with respect to inclusion.
An \emph{(ideal) triangulation} of $S$ is a maximal multiarc.
An Euler characteristic argument shows that every triangulation has $E(S)$ arcs and $F(S)$ faces.
Each complementary region of a triangulation is a triangle with vertices on the marked points of $S$.
Triangles have embedded interior, but their boundary may be non-embedded.
In particular, triangles can be \emph{folded} as shown in Figure~\ref{fig:folded}.

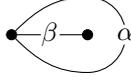
\begin{figure}[ht]
\centering
\begin{tikzpicture}[scale=0.5, rotate=0]

\node [dot] (a) at (0,0) {};
\node [dot] (b) at (2,0) {};
\coordinate (l) at (2,1);
\coordinate (r) at (2,-1);
\coordinate (t) at (3,0);

\draw (a) to [out=45,in=180] (l) to [out=0,in=90] (t) node {\contour*{white}{$\alpha$}} to [out=270,in=0] (r) to [out=180,in=-45] (a);
\draw (a) -- node {\contour*{white}{$\beta$}} (b);

\end{tikzpicture}
\caption{A folded triangle.}
\label{fig:folded}
\end{figure}





\subsubsection{The flip graph}
The \emph{flip graph} $\flip(S)$ has a vertex for each triangulation.
Triangulations $\calT, \calT' \in \flip(S)$ are connected via an edge (of length one) if and only if they differ by a \emph{flip}.
This move consists of replacing the diagonal of a quadrilateral inside the triangulation with the other diagonal, as shown in Figure~\ref{fig:flip}.

Any arc of a triangulation is either flippable or appears as the arc $\beta$ in Figure~\ref{fig:folded}.
In the latter case, the arc $\beta$ is flippable after first flipping the arc $\alpha$.
Thus for any arc $\alpha$ there is a triangulation $\calT \ni \alpha$ in which $\alpha$ is flippable.

\begin{figure}[htb]
\centering
\begin{tikzpicture}[scale=1.25,thick]

\begin{scope}[shift={(-1.5,0)}]
	\foreach \i in {45,135,...,315} {\draw (\i:1) -- (\i+90:1);}
	\foreach \i in {45,135,...,315} {\node [dot] at (\i:1) {};}
	\draw (45:1) -- node {\contour*{white}{$\alpha$}} (225:1);
	\node at (0,-1) {$\calT$};
\end{scope}

\begin{scope}[shift={(1.5,0)}]
	\foreach \i in {45,135,...,315} {\draw (\i:1) -- (\i+90:1);}
	\foreach \i in {45,135,...,315} {\node [dot] at (\i:1) {};}
	\draw (135:1) -- (315:1);
	\node at (0,-1) {$\calT'$};
\end{scope}

\draw [->] (-0.5,0) -- node[above] {Flip} (0.5,0);

\end{tikzpicture}
\caption{Flipping the arc $\alpha$ of a triangulation.}
\label{fig:flip}
\end{figure}
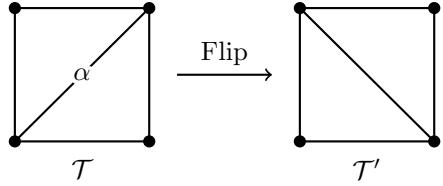
The flip graph appears implicitly in works of Harer \cite{Harer}, Mosher \cite{Mosher} and Penner \cite{Penner}.
It is connected \cite{Harer} and the mapping class group acts on it \emph{geometrically}, that is, properly, cocompactly and by isometries.
Thus by the \v{S}varc--Milnor Lemma \cite[Proposition~I.8.19]{BridsonHaefliger} this graph is quasi-isometric to $\EMod(S)$.
The geometry of the flip graph was recently studied by the second author and Parlier in \cite{DisarloParlier}. 
By \cite[Theorem 1.2]{KorkmazPapadopoulos} and \cite[Theorem 1.1]{AramayonaKoberdaParlier}, we have: 

\begin{theorem}
\label{thrm:flip_rig}
The natural homomorphism
\[ \EMod(S) \to \Aut(\flip(S)) \]
is an isomorphism. \qed
\end{theorem}


\subsubsection{The arc graph}

The arc graph $\calA(S)$ has a vertex for each arc on $S$.
Arcs $\alpha, \beta \in \calA(S)$ are connected via an edge (of length one) if and only if they are disjoint.
The arc graph appeared first in work of Harer \cite{Harer}.
The arc graph of $S$ can be naturally extended to a flag simplicial complex called the \emph{arc complex} of $S$.
The one-skeleton of the dual of the arc complex is the flip graph.

\begin{theorem}
\label{thrm:arc_isom}
The graphs $\calA(S)$ and $\calA(S')$ are isomorphic if and only if $S$ and $S'$ are homeomorphic.
\end{theorem}

\begin{proof}
Every isomorphism between $\calA(S)$ and $\calA(S')$ induces an isomorphism between $\flip(S)$ and $\flip(S')$.
If neither $S$ nor $S'$ is the complement of a (possibly empty) multiarc on the four-times marked sphere or the twice-marked torus then the result follows from \cite[Theorem~1.4]{AramayonaKoberdaParlier}.
Otherwise:
\begin{itemize}
\item If $\partial S = \emptyset = \partial S'$ then the result follows from \cite[Theorem~1.1]{KorkmazPapadopoulos}.
\item If $\partial S = \emptyset \neq \partial S'$ then $E(S) = 6$ and $E(S') \leq 5$.
Since maximal complete subgraphs of $\calA(S)$ have exactly $E(S)$ vertices, it follows that $\calA(S)$ and $\calA(S')$ cannot be isomorphic.
\item If $\partial S \neq \emptyset \neq \partial S'$ then the result follows from \cite[Theorem~1.1]{Disarlo}.
\end{itemize}
In any case, the result holds.
\end{proof}

The mapping class group acts on $\calA(S)$ cocompactly and by isometries but not properly. Generally, the automorphism group of $\calA(S)$ is isomorphic to $\EMod(S)$ by \cite[Theorem 1.2]{IrmakMcCarthy} and  \cite[Theorem~1.2]{Disarlo}.

\subsection{The polygonalisation complex}

A \emph{polygonalisation} $P$ (or \emph{ideal cell decomposition}) is a multiarc such that each complementary region of $P$ is a \emph{polygon}, a topological disk with at least three sides and no marked points in its interior.
These are dual to fat graphs (also known as ribbon graphs) \cite[Section~4.2]{ReshetikhinTuraev}.

\begin{remark}
\label{rem:polygon_containment}
If $P$ is a polygonalisation and $Q$ is a multiarc such that $P \subseteq Q$ then $Q$ is a polygonalisation.
Additionally, suppose that $P, Q$ and $R$ are polygonalisations:
\begin{itemize}
\item If $P, Q \subseteq R$ then $P \cup Q$ is a polygonalisation.
\item If $P, Q \supseteq R$ then $P \cap Q$ is a polygonalisation.
\end{itemize}
\end{remark}

\begin{definition}
The \emph{polygonalisation complex} $\pol(S)$ is a cube complex in which vertices correspond to polygonalisations.
Two polygonalisations $P, Q \in \pol(S)$ are connected by an edge (of length one) if and only if they differ by a single arc, that is, their symmetric difference $P \Delta Q$ is a single arc.
Inductively, a $k$--cube is added whenever its $(k-1)$--skeleton appears.
See Appendix~\ref{sec:examples} for examples.
\end{definition}

Since any polygonalisation can be extended to a triangulation and $\flip(S)$ is connected, $\pol(S)$ is also connected.
Again, $\EMod(S)$ acts geometrically on $\pol(S)$ and so they are quasi-isometric.

We now establish some definitions and notation that will be used throughout.

\begin{definition}
For an arc $\alpha \in \calA(S)$, its \emph{stratum} $\pol_\alpha(S)$ is the subcomplex of $\pol(S)$ induced by the subset $\{ P \in \pol(S) : \alpha \in P \}$.
Similarly, define $\overline{\pol_\alpha}(S)$ to be the subcomplex induced by $\{ P \in \pol(S) : \alpha \notin P \}$.
\end{definition}

\begin{definition}
An arc $\alpha$ is \emph{removable from $P$} if $\alpha \in P$ and $P - \{\alpha\}$ is also a polygonalisation.
An arc $\alpha$ \emph{addable to $P$} if $\alpha \notin P$ and $P \cup \{\alpha\}$ is a polygonalisation.
\end{definition}

Observe that $\alpha \in P$ is removable if and only if its interior meets two distinct polygons of $P$.
Let $\partial \pol_\alpha(S)$ (resp. $\partial \overline{\pol_\alpha}(S)$) be the subcomplex of $\pol_\alpha(S)$ (resp.\ $\overline{\pol_\alpha}(S)$) induced by the polygonalisations in which $\alpha$ is removable (resp.\ addable).

\begin{notation}
For $P, Q \in \pol(S)$, write $P \succ Q$ if $P \supseteq Q$ and $|P - Q| = 1$.
\end{notation}

For an edge $e = \{P, Q\}$ in $\pol(S)$ let $\arc(e) \defeq P \Delta Q \in \calA(S)$ denote the arc that appears in $P$ but not $Q$ (or vice versa).

\begin{remark}
\label{rem:arc}
For any edge $e = \{P, Q\}$ in $\pol(S)$, observe that $\arc(e) = \alpha$ if and only if $P \in \partial \pol_\alpha(S)$ and $Q \in \partial \overline{\pol_\alpha}(S)$ (or vice versa).
\end{remark}

\section{The cube complex structure of \texorpdfstring{$\pol(S)$}{P(S)}}

We describe some of the key properties of the cube complex structure of $\pol(S)$.
We begin by recalling some of the standard terms for cube complexes.
For a complete reference see \cite{Wise}.

An $n$--cube is the Euclidean cube $[-\frac{1}{2}, \frac{1}{2}]^n$.
For every $0 \leq k \leq n - 1$ a $k$--\emph{face} is a subspace obtained restricting $n - k$ coordinates to $\pm \frac{1}{2}$.
A $k$--face is also a $k$--cube.
A \emph{cube complex} $X$ is a cell complex obtained by gluing cubes along their faces by isometries.

The \emph{link} of a vertex $v$ of $X$ is the complex $\lnk(v)$ induced on a small sphere about $v$ by $X$.
A \emph{flag complex} is a simplicial complex where $n+1$ vertices span an $n$--simplex if and only if they are pairwise adjacent.
A cube complex $X$ is \emph{non-positively curved} if it satisfies \emph{Gromov's link condition}: The link of each vertex is a flag complex \cite[Section~4.2.C]{Gromov}.
If $X$ is non-positively curved and simply connected then $X$ is CAT(0).

A \emph{midcube} is a subspace of a cube obtained by restricting exactly one coordinate to $0$.
A \emph{hyperplane} is the union of midcubes that meet parallel edges.
The \emph{carrier} $N(H)$ of a hyperplane $H$ is the union of all the cubes such that intersect $H$ in a midcube.
Sageev showed that CAT(0) cube complexes come equipped with a family of ``nice'' hyperplanes.

\begin{theorem}[{\cite[Theorem~1.1]{Sageev}}]
If $X$ is a CAT(0) cube complex then:
\begin{enumerate}
\item each midcube lies in an embedded hyperplane;
\item every hyperplane $H$ is \emph{two-sided}, that is, $N(H) \isom [-\frac{1}{2}, \frac{1}{2}] \times H$;
\item every hyperplane $H$ separates $X$, that is, $X - H$ consists of two connected components $H^+$ and $H^-$ called \emph{halfspaces};
\item every hyperplane $H$ is a CAT(0) cube complex; and
\item every hyperplane $H$ and its carrier $N(H)$ are convex in $X$.
\end{enumerate}
\end{theorem}

Despite the fact that generically $\pol(S)$ is not CAT(0), we show that Properties 1, 2 and 3 still hold for $\pol(S)$ (Theorem~\ref{thrm:pol_sageev}).
However, generically, Properties 4 and 5 do not hold for $\pol(S)$. 
For example, see Table~\ref{tab:example_hyperplanes} and $S_{0,0}^{2,1}$ in Table~\ref{tab:examples} respectively. 
\subsection{The cubes of \texorpdfstring{$\pol(S)$}{P(S)}}

The following lemma will also be useful in Section~\ref{sec:rigidity}.

\begin{lemma}[Square lemma]
\label{lem:square_lemma}
Suppose that $P_1, P_2, P_3, P_4 \subseteq \pol(S)$ is an embedded $4$--cycle.
If $P_1 \succ P_2$ then $P_4 \succ P_3$.
Moreover, $\arc(\{P_1, P_2\}) = \arc(\{P_3, P_4\})$.
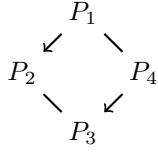
\begin{center}
\begin{tikzpicture}[scale=0.8, thick]
\node (T) at (0,1) {$P_1$};
\node (Ta) at (-1,0) {$P_2$};
\node (Tb) at (1,0) {$P_4$};
\node (Tab) at (0,-1) {$P_3$};
\draw [->] (T) -- (Ta);
\draw (T) -- (Tb);
\draw (Ta) -- (Tab);
\draw [->] (Tb) -- (Tab);
\end{tikzpicture}
\end{center}
\end{lemma}

\begin{proof}
Suppose that instead $P_3 \succ P_4$.
By considering the cardinalities of the $P_i$'s, we see that $P_1 \succ P_2 \prec P_3 \succ P_4 \prec P_1$.
Since $P_2 \neq P_4$, by Remark~\ref{rem:polygon_containment} we have $P_1 \supseteq P_2 \cup P_4 \supsetneq P_2$.
Since $|P_1| = |P_2| + 1$, we deduce $P_1 = P_2 \cup P_4$.
Similarly, $P_3 = P_2 \cup P_4$ contradicting the fact that this cycle is embedded.
Furthermore, if $\arc(\{P_1, P_2\}) \neq \arc(\{P_3, P_4\})$ then counting the number of components shows that $P_2 = P_4$.
Again, this contradicts this cycle being embedded.
\end{proof}

The $n$--\emph{cube graph} $C_n$ is the 1--skeleton of the standard $n$--cube.
In particular, an embedded $4$--cycle is a $C_2$.
By induction on $n$, the square lemma shows that any embedded cube graph $C$ in $\pol(S)$ contains a unique \emph{source} $C^+$ and a unique \emph{sink} $C^-$.
That is, for any polygonalisation $P$ in $C$ we have that $C^- \subseteq P \subseteq C^+$.

For polygonalisations $P \subseteq Q$, let
\[ [P, Q] \defeq \{R \in \pol(S) \; : \; P \subseteq R \subseteq Q\}. \]
In fact these are all the cubes that appear in $\pol(S)$:

\begin{lemma}[Characterisation of cubes]
\label{lem:cube_char}
If $P \subseteq Q$ then $[P, Q]$ forms the vertex set of an $n$--cube in $\pol(S)$ where $n = |P| - |Q|$.
Conversely, if $C$ is an embedded $n$--cube in $\pol(S)$ then the vertex set of $C$ is $[C^-, C^+]$. \qed
\end{lemma}

When $\partial S = \emptyset$, this characterisation of cubes enables us to apply Penner's argument for the contractibility of the fatgraph complex \cite[Theorem~2.5]{Penner} to $\pol(S)$.
Hence, in this case the polygonalisation complex $\pol(S)$ is contractible.

\subsection{The hyperplanes of \texorpdfstring{$\pol(S)$}{P(S)}}

Let $\sim$ denote the equivalence relation on the edges of $\pol(S)$ generated by $e \sim e'$ if and only if $e$ and $e'$ are opposite edges of some square in $\pol(S)$.
We say that $e$ and $e'$ are \emph{parallel} if $e \sim e'$.

\begin{definition}[{\cite[Definition~1.4]{Sageev}}]
A \emph{hyperplane} is the set of midcubes of $\pol(S)$ that meet edges in $[e]$, for some equivalence class of edges.
We write $\Hyp(\pol(S))$ for the set of all hyperplanes of $\pol(S)$.
\end{definition}

By the square lemma, if edges $e, e'$ are parallel then $\arc(e) = \arc(e')$.
Thus, the $\arc$ map descends to a well-defined map $\arc \from \Hyp(\pol(S)) \to \calA(S)$.
This map is surjective: for any $\alpha \in \calA(S)$ there is a triangulation $\calT$ from which $\alpha$ is removable and so $\arc(\{\calT, \calT - \{\alpha\} \}) = \alpha$.
One of the aims of this section is to prove that this map is also injective.
To achieve this we shall require some technical results.

\begin{lemma}[Pentagon detour lemma]
\label{lem:pentagon_detour}
Let $\calT$ and $\calT'$ be adjacent triangulations in $\flip(S)$ and $Q \defeq \calT \cap \calT'$.
Suppose there is an arc $\alpha \in Q$ that is removable from $\calT$ and $\calT'$ but not from $Q$.
Then there is a path $\calT = \calT_0, \calT_1, \calT_2, \calT_3, \calT_4 = \calT'$ such that $\alpha$ is removable from each $\calT_i$ and from each $Q_i \defeq \calT_{i-1} \cap \calT_i$.
\end{lemma}

\begin{proof}
Let $\beta \defeq \arc(\{\calT, Q\})$.
Since $\alpha$ is not removable from $Q$ but is removable from $\calT$ and $\calT'$, it must appear as opposite sides of the unique square of $Q$.
Since $\calT$ has at least three triangles, this square has another side $\gamma$ that is removable from $Q$.
The arcs $\beta$ and $\gamma$ are therefore a pair of chords of a pentagon in $\calT$, as shown in Figure~\ref{fig:pentagon_detour}.
The path $\calT = \calT_0, \calT_1, \calT_2, \calT_3, \calT_4 = \calT'$ is the one obtained going around the natural five-cycle in $\flip(S)$ defined by $\beta$ and $\gamma$ in the other direction.
\end{proof}

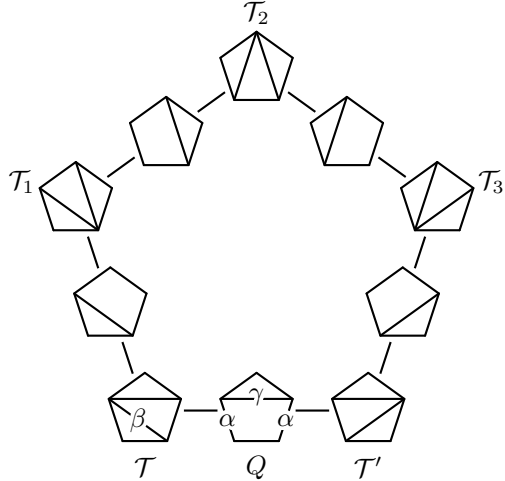
\begin{figure*}[ht]
	\centering
	\begin{tikzpicture}[scale=2,thick]

\pgfmathsetmacro{\inner}{0.25}
\pgfmathsetmacro{\outer}{1.25}

\foreach \i in {90,162,...,378} {\draw (\i:\outer) -- (\i+72:\outer);}

\foreach \i in {90,162,...,378} {
	\begin{scope}[shift={(\i:\outer)}]
		\draw [fill,white] (0,0) circle (\inner);
		\foreach \j in {90,162,...,378} {\draw [line cap=round] (\j:\inner) -- (\j + 72:\inner);}
		\draw (-2*\i+144-90:\inner) -- (-2*\i-90:\inner);
		\draw (-2*\i-90:\inner) -- (-2*\i-144-90:\inner);
	\end{scope}
	\begin{scope}[shift={($(\i:\outer)!0.5!(\i+72:\outer)$)}]
		\draw [fill,white] (0,0) circle (\inner);
		\foreach \j in {90,162,...,378} {\draw [line cap=round] (\j:\inner) -- (\j + 72:\inner);}
		\draw (-2*\i-90:\inner) -- (-2*\i-144-90:\inner);
	\end{scope}
}

\foreach \i in {1,2,3} {\node at (90+72+72-72*\i:\outer+1.5*\inner) {$\calT_\i$};}

\coordinate (A) at ($(90+72+72:\outer)$);
\coordinate (B) at ($(90+72+72+72:\outer)$);
\coordinate (C) at ($(A)!0.5!(B)$);
\node at ($(A)+(0,-1.5*\inner)$) {$\calT$};
\node at ($(C)+(0,-1.5*\inner)$) {$Q$};
\node at ($(B)+(0,-1.5*\inner)$) {$\calT'$};

\node at ($(C)+0.5*(90+72:\inner)+0.5*(90+72+72:\inner)$) {\contour*{white}{$\alpha$}};
\node at ($(C)+0.5*(90+72+72+72:\inner)+0.5*(90+72+72+72+72:\inner)$) {\contour*{white}{$\alpha$}};
\node at ($(C)+0.5*(90+72:\inner)+0.5*(90-72:\inner)+(0,0.015)$) {\contour*{white}{$\gamma$}};
\node at ($(A)+0.5*(90+72:\inner)+0.5*(90-72-72:\inner)$) {\contour*{white}{$\beta$}};

\end{tikzpicture}
	\caption{A detour around a pentagon to stay in $\partial \pol_\alpha(S)$.}
	\label{fig:pentagon_detour}
\end{figure*}

\begin{proposition}
\label{prop:posboundary_connected}
For each arc $\alpha \in \calA(S)$, the subcomplex $\partial \pol_\alpha(S)$ is connected.
\end{proposition}

\begin{proof}
Suppose $P, Q \in \partial \pol_\alpha(S)$ are polygonalisations and let $\calT$ and $\calT'$ be triangulations containing $P$ and $Q$ respectively.
Note that $\alpha$ is removable from $\calT$ and $\calT'$.
Furthermore, $P$ and $\calT$ are in the same path-component of $\partial \pol_\alpha(S)$ (and likewise for $Q$ and $\calT'$).
Thus to show $\partial \pol_\alpha(S)$ is connected it suffices to show that $\calT$ and $\calT'$ are in the same path-component.

There is a path $\calT = \calT_0, \calT_1, \ldots, \calT_n = \calT'$ in $\flip(S)$ such that $\alpha \in \calT_i$ \cite[Corollary~2.15]{DisarloParlier}.
Call $\calT_i$ \emph{good} if $\alpha$ can be removed from it and \emph{bad} otherwise.
We describe how to modify this path to avoid any bad triangulations.

Suppose that $\calT_i$ is the first bad triangulation in this path.
Note that $0 < i < n$ since both $\calT_0$ and $\calT_n$ are good.
Let $\beta$ and $\gamma$ be the arcs of $\calT_i$ that can be flipped to obtain $\calT_{i-1}$ and $\calT_{i+1}$ respectively.
There are three possibilities to consider:
\begin{itemize}
	\item If $\beta \cup \gamma$ is supported on two triangles then, since $\beta$ cuts off a once-marked monogon, $\beta = \gamma$.
		Thus $\calT_{i-1} = \calT_{i+1}$ and so we simplify our path by removing $\calT_i$ and $\calT_{i+1}$.
	\item If $\beta \cup \gamma$ is supported on three triangles then they fill a pentagon in $\calT_i$.
		Thus we may replace $\calT_i$ with the good triangulations $\calT', \calT''$ as shown in Figure~\ref{fig:pentagon}.
	\item If $\beta \cup \gamma$ is supported on four triangles then the flips commute.
		Thus we may replace $\calT_i$ with the good triangulation $\calT'$ as shown in Figure~\ref{fig:square}.
\end{itemize}
Performing any of these modifications will reduce the number of bad triangulations.
Hence, by induction, we may assume $\alpha$ is removable from each $\calT_i$ along this path.
Taking the barycentric subdivision gives a path $\calT_0, Q_1, \calT_1, Q_2, \ldots, \calT_{n-1}, Q_n, \calT_n$ in $\pol_\alpha(S)$, where $Q_i \defeq \calT_{i-1} \cap \calT_i$.
Replacing subpaths with pentagon detours if necessary, we may assume that $\alpha$ is also removable from each $Q_i$.
This yields a path in $\partial \pol_\alpha(S)$ connecting $\calT$ to $\calT'$ as required.
\end{proof}

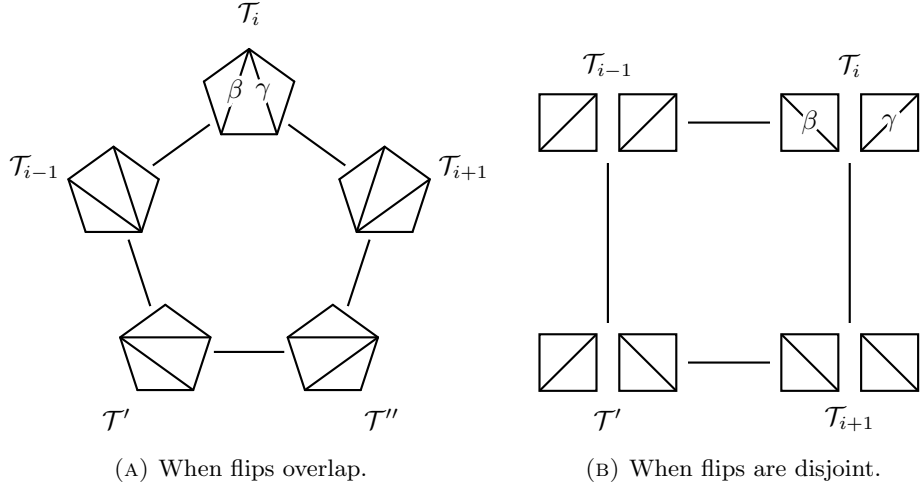
\begin{figure}[ht]
	\centering
	\begin{subfigure}[b]{0.5\textwidth}
		\centering
		\begin{tikzpicture}[scale=1.25,thick,baseline=0]

\pgfmathsetmacro{\inner}{0.5}
\pgfmathsetmacro{\outer}{1.5}

\foreach \i in {90,162,...,378} {\draw (\i:\outer) -- (\i+72:\outer);}

\foreach \i in {90,162,...,378} {
	\begin{scope}[shift={(\i:\outer)}]
		\draw [fill,white] (0,0) circle (\inner);
		\foreach \j in {90,162,...,378} {\draw [line cap=round] (\j:\inner) -- (\j + 72:\inner);}
		\draw (-2*\i+144-90:\inner) -- (-2*\i-90:\inner);
		\draw (-2*\i-90:\inner) -- (-2*\i-144-90:\inner);
	\end{scope}
}

\node at (90+72+72:\outer+1.75*\inner) {$\calT'$};
\node at (90+72:\outer+1.75*\inner) {$\calT_{i-1}$};
\node at (90:\outer+1.75*\inner) {$\calT_i$};
\node at (90-72:\outer+1.75*\inner) {$\calT_{i+1}$};
\node at (90-72-72:\outer+1.75*\inner) {$\calT''$};

\node at ($(90:\outer)+0.5*(90:\inner)+0.5*(90+72+72:\inner)$) {\contour*{white}{$\beta$}};
\node at ($(90:\outer)+0.5*(90:\inner)+0.5*(90-72-72:\inner)$) {\contour*{white}{$\gamma$}};

\end{tikzpicture}
		\caption{When flips overlap.}
		\label{fig:pentagon}
	\end{subfigure}%
	~
	\begin{subfigure}[b]{0.5\textwidth}
		\centering
		\begin{tikzpicture}[scale=1.5,thick,baseline=0]

\pgfmathsetmacro{\inner}{0.5}
\pgfmathsetmacro{\outer}{1.5}

\foreach \i in {45,135,...,315} {\draw (\i:\outer) -- (\i+90:\outer);}

\foreach \i in {45,135,...,315} {
	\begin{scope}[shift={(\i:\outer)}]
		\fill [white] (-0.2-\inner,-0.1-0.5*\inner) -- (0.2+\inner,-0.1-0.5*\inner) -- (0.2+\inner,0.1+0.5*\inner) -- (-0.2-\inner,0.1+0.5*\inner) -- cycle;
		\draw (0.1,-0.5*\inner) -- (0.1+\inner,-0.5*\inner) -- (0.1+\inner,0.5*\inner) -- (0.1,0.5*\inner) -- cycle;
		\draw (-0.1,-0.5*\inner) -- (-0.1-\inner,-0.5*\inner) -- (-0.1-\inner,0.5*\inner) -- (-0.1,0.5*\inner) -- cycle;
	\end{scope}
}

\draw ($(45:\outer)+(-0.1,-0.5*\inner)$) -- ($(45:\outer)+(-0.1-\inner,0.5*\inner)$);
\draw ($(45:\outer)+(0.1,-0.5*\inner)$) -- ($(45:\outer)+(0.1+\inner,0.5*\inner)$);
\node at ($(45:\outer)+(0,\inner)$) {$\calT_{i}$};
\node at ($(45:\outer)+(-0.1-0.5*\inner,0)$) {\contour*{white}{$\beta$}};
\node at ($(45:\outer)+(0.1+0.5*\inner,0)$) {\contour*{white}{$\gamma$}};

\draw ($(135:\outer)+(-0.1,0.5*\inner)$) -- ($(135:\outer)+(-0.1-\inner,-0.5*\inner)$);
\draw ($(135:\outer)+(0.1+\inner,0.5*\inner)$) -- ($(135:\outer)+(0.1,-0.5*\inner)$);
\node at ($(135:\outer)+(0,\inner)$) {$\calT_{i-1}$};

\draw ($(225:\outer)+(-0.1-\inner,-0.5*\inner)$) -- ($(225:\outer)+(-0.1,0.5*\inner)$);
\draw ($(225:\outer)+(0.1+\inner,-0.5*\inner)$) -- ($(225:\outer)+(0.1,0.5*\inner)$);
\node at ($(225:\outer)+(0,-\inner)$) {$\calT'$};

\draw ($(315:\outer)+(-0.1-\inner,0.5*\inner)$) -- ($(315:\outer)+(-0.1,-0.5*\inner)$);
\draw ($(315:\outer)+(0.1,0.5*\inner)$) -- ($(315:\outer)+(0.1+\inner,-0.5*\inner)$);
\node at ($(315:\outer)+(0,-\inner)$) {$\calT_{i+1}$};

\end{tikzpicture}
		\caption{When flips are disjoint.}
		\label{fig:square}
	\end{subfigure}
	\caption{Diverting a path around a bad triangulation.}
\end{figure}

Observe that $P, Q \in \partial \overline{\pol_\alpha}(S)$ are adjacent if and only if they form a square with $P \cup \{\alpha\}, Q \cup \{\alpha\} \in \partial \pol_\alpha(S)$.
Hence, by the above proposition, $\partial \overline{\pol_\alpha}(S)$ is also connected.

\begin{lemma}
\label{lem:strata_connected}
Let $\alpha \in \calA(S)$ be an arc.
Then the subcomplexes $\pol_\alpha(S)$ and $\overline{\pol_\alpha}(S)$ in $\pol(S)$ are both connected.
\end{lemma}

\begin{proof}
Suppose that $P \in \pol_\alpha(S)$ and $Q \in \overline{\pol_\alpha}(S)$.
For any edge-path in $\pol(S)$ from $P$ to $Q$, the arc $\alpha$ must be removed from a polygonalisation at some point along this path.
Hence such a path in $\pol(S)$ must cross $H_\alpha$.
Therefore there is a path within $\pol_\alpha(S)$ from $P$ to $\partial \pol_\alpha(S)$ and a path within $\overline{\pol_\alpha}(S)$ from $Q$ to $\partial \overline{\pol_\alpha}(S)$.
Since $\partial \pol_\alpha(S)$ and $\partial \overline{\pol_\alpha}(S)$ are both connected, $\pol_\alpha(S)$ and $\overline{\pol_\alpha}(S)$ are too.
\end{proof}

\begin{theorem}
\label{thrm:pol_sageev}
The map $\arc \from \Hyp(\pol(S)) \to \calA(S)$ is a bijection.
Moreover, for each arc $\alpha \in \calA(S)$, the corresponding hyperplane $H_\alpha$ is embedded, two-sided and separates $\pol(S)$ into two connected components: $\pol_\alpha(S)$ and $\overline{\pol_\alpha}(S)$.
\end{theorem}

\begin{proof}
To show that the map $\arc \from \Hyp(\pol(S)) \to \calA(S)$ is a bijection we must prove that edges $e, e'$ in $\pol(S)$ are parallel if and only if $\arc(e) = \arc(e')$.
As described above, the forwards direction holds trivially by the square lemma.

For the backwards direction, suppose that $e, e'$ are edges of $\pol(S)$ such that $\arc(e) = \arc(e') = \alpha$.
By Remark~\ref{rem:arc}, we deduce that $e = \{P, Q\}$ and $e' = \{P', Q'\}$ for some $P, P' \in \partial \pol_\alpha(S)$ and $Q, Q' \in \partial \overline{\pol_\alpha}(S)$.
By Proposition~\ref{prop:posboundary_connected} there is a path $p$ from $P$ to $P'$ in $\partial \pol_\alpha(S)$.
Now $\alpha$ can be removed from each polygonalisation in $p$ to obtain a parallel path $q$ from $Q$ to $Q'$ in $\partial \overline{\pol_\alpha}(S)$.
This gives a sequence of squares from which we can deduce that $e$ and $e'$ are parallel.

It is easy to verify that hyperplanes are embedded.
If a hyperplane self-intersects then a self-intersection must occur in some square.
Then all four edges in this square must correspond to adding/removing the same arc, which is impossible.

As in the proof of Lemma~\ref{lem:strata_connected}, any path from a vertex in $\pol_\alpha(S)$ to a vertex in $\overline{\pol_\alpha}(S)$ must cross $H_\alpha$, hence $H_\alpha$ separates $\pol_\alpha(S)$ and $\overline{\pol_\alpha}(S)$.
By Lemma~\ref{lem:strata_connected}, these are both connected and so $H_\alpha$ separates $\pol(S)$ into $\pol_\alpha(S)$ and $\overline{\pol_\alpha}(S)$.

Finally, each hyperplane must be two-sided since it is separating.
\end{proof}

\begin{corollary}
Let $\alpha \in \calA(S)$ be an arc.
Then the hyperplane $H_\alpha$ separates polygonalisations $P, Q \in \pol(S)$ if and only if $\alpha \in P \Delta Q$.
In particular, there are at most $2 E(S)$ hyperplanes separating any given pair of polygonalisations. \qed
\end{corollary}

\begin{remark}
The hyperplanes of $\pol(S)$ cannot self-osculate, but they can interosculate.
For example, see $S_{0,0}^{2,1}$ in Table~\ref{tab:examples}.
\end{remark}

\subsection{The crossing graph of \texorpdfstring{$\pol(S)$}{P(S)}}

\begin{definition}
\label{def:hyp_cross}
Two hyperplanes $H_\alpha$ and $H_\beta$ \emph{cross}, denoted $H_\alpha \crosses H_\beta$, if
\[ \pol_\alpha(S) \cap \pol_\beta(S), \;
\pol_\alpha(S) \cap \overline{\pol_\beta}(S), \;
\overline{\pol_\alpha}(S) \cap \pol_\beta(S) \; \textrm{and} \;
\overline{\pol_\alpha}(S) \cap \overline{\pol_\beta}(S) \]
are all non-empty.
\end{definition}

There are several ways of characterising when hyperplanes cross.

\begin{proposition}
\label{prop:equivalent_crossing}
The following are equivalent:
\begin{enumerate}
\item The hyperplanes $H_\alpha$ and $H_\beta$ cross.
\item The arcs $\alpha$ and $\beta$ are distinct and disjoint but do not form a folded triangle.
\item There is a triangulation $\calT$ containing $\alpha \neq \beta$ such that $\calT - \{\alpha, \beta\}$ is a polygonalisation.
\end{enumerate}
\end{proposition}

\begin{proof}
We follow a cycle of implications:

$1 \implies 2$:
There is a polygonalisation $P$ containing $\alpha$ and $\beta$ and so these arcs must be disjoint.
If $\alpha$ and $\beta$ form a folded triangle then every polygonalisation that contains $\alpha$ must contain $\beta$ or vice versa.
Therefore either $\pol_\alpha(S) \cap \overline{\pol_\beta}(S)$ or $\overline{\pol_\alpha}(S) \cap \pol_\beta(S)$ is empty and so $H_\alpha$ and $H_\beta$ do not cross.

$2 \implies 3$:
Since $\alpha$ and $\beta$ do not form a folded triangle, there is a triangulation $\calT$ containing $\alpha$ and $\beta$ in which both are removable.
Let $Q \defeq \calT - \{\alpha\}$.
Suppose that $\beta$ is not removable from $Q$.
Then, as in the proof of Lemma~\ref{lem:pentagon_detour}, a polygon of $Q$ is a square, two sides of which are $\beta$.
Since $\calT$ has at least three triangles, this square has another side $\gamma$ that is removable from $Q$.
We flip $\gamma$ in $\calT$ to obtain a new triangulation $\calT'$ in which both $\alpha$ and $\beta$ are simultaneously removable.

$3 \implies 1$:
The polygonalisations $\calT$, $\calT - \{\beta\}$, $\calT - \{\alpha\}$ and $\calT - \{\alpha, \beta\}$ show that the four sets required by Definition~\ref{def:hyp_cross} are all non-empty.
\end{proof}

\begin{definition}
The \emph{crossing graph} $\Cr(\pol(S))$ is the graph with a vertex for each hyperplane in $\pol(S)$.
Two hyperplanes $H, H' \in \Cr(\pol(S))$ are connected via an edge (of length one) if and only if $H \crosses H'$.
\end{definition}

By Proposition~\ref{prop:equivalent_crossing}, we immediately deduce that the crossing graph $\Cr(\pol(S))$ embeds into the arc graph $\calA(S)$.
To get control over this embedding we will need to consider the paths in $\calA(S)$ \emph{without folds}, that is, the ones in which no consecutive pair of arcs form a folded triangle.

\begin{lemma}
\label{lem:geodesic}
If $d(\alpha, \beta) = n \geq 3$ then there is a geodesic
\[ \alpha = \alpha_0, \alpha_1, \ldots, \alpha_n = \beta \]
in $\calA(S)$ that is without folds.
\end{lemma}

\begin{proof}
Let
\[ \alpha = \alpha_0, \alpha_1, \ldots, \alpha_n = \beta \]
be a geodesic in $\calA(S)$ from $\alpha$ to $\beta$.

Suppose that $\{\alpha_i, \alpha_{i+1}\}$ form a folded triangle for some $0 < i < n-1$.
Then either $\intersection(\alpha_{i-1}, \alpha_{i+1}) = 0$ or $\intersection(\alpha_i, \alpha_{i+2}) = 0$, which contradicts this path being a geodesic.
Hence we need only consider folded triangles formed by $\{\alpha_0, \alpha_1\}$ and $\{\alpha_{n-1}, \alpha_n\}$.
We show that if the former occurs then there are arcs $\alpha_1', \alpha_2'$ such that
\[ \alpha_0, \alpha_1', \alpha_2', \alpha_3 \]
is a geodesic without folds.
We will take care to ensure that $\alpha_1'$ and $\alpha_2'$ are both arcs, that is, that they are not null-homotopic nor boundary-parallel.
Similarly we can replace the end of the geodesic if $\{\alpha_{n-1}, \alpha_n\}$ forms a folded triangle.

Now note that if $\{\alpha_0, \alpha_1\}$ forms a folded triangle then $\alpha_0$ must cut off a once-marked monogon.
Furthermore, $\alpha_2$ must have at least one endpoint on $x$, the inner marked point of this monogon, since it must intersect $\alpha_0$.

If $\alpha_2$ has exactly one endpoint on $x$ then let $\alpha_1'$ and $\alpha_1''$ be as shown in Figure~\ref{fig:geodesic_start_one}.
If $\alpha_1'$ or $\alpha_1''$ is null-homotopic then $\alpha_1 = \alpha_2$, which contradicts this path being a geodesic.
If $\alpha_1'$ and $\alpha_1''$ are both boundary-parallel then $F(S) = 2$, which is again a contradiction.
Hence without loss of generality $\alpha_1'$ is an arc.
Since this arc is disjoint from $\alpha_0$ and $\alpha_2$ and does not form a folded triangle with either, $\alpha_0, \alpha_1', \alpha_2, \alpha_3$ forms the required geodesic without folds.

On the other hand, if $\alpha_2$ has two endpoints on $x$, we construct $\alpha_1'$ by surgery as in Figure~\ref{fig:geodesic_start_two}.
We note that $\alpha_1'$ cannot be null-homotopic.
If it were then $\alpha_2$ cuts off a once-marked monogon and so $\alpha_3$ must also be disjoint from $\alpha_1$ as it is disjoint from $\alpha_2$.
Again, this contradicts this path being a geodesic.
Hence we need only consider the case in which $\alpha_1'$ is boundary-parallel, since otherwise $\alpha_1'$ is an arc and $\alpha_0, \alpha_1', \alpha_2, \alpha_3$ is a geodesic without folds.

If $\alpha_1'$ is boundary-parallel then $\alpha_2$ must cut off an annulus with one marked point on each boundary component.
Since it must meet $\alpha_1$, the arc $\alpha_3$ must be contained in the annulus cut off by $\alpha_2$.
However, as $\alpha_2$ is essential, there is an arc $\alpha_2'$ in the other connected component with one endpoint on $x$.
By construction $\alpha_2'$ does not cut off an annulus and it is disjoint from $\alpha_1, \alpha_2, \alpha_3$.
Hence $\alpha_0, \alpha_1, \alpha_2', \alpha_3$ is a geodesic and $\alpha_2'$ has a single endpoint on $x$.
Thus we can repeat the above argument to find a new arc $\alpha_1'$ such that
\[ \alpha_0, \alpha_1', \alpha_2', \alpha_3 \]
is a geodesic without folds.
\end{proof}

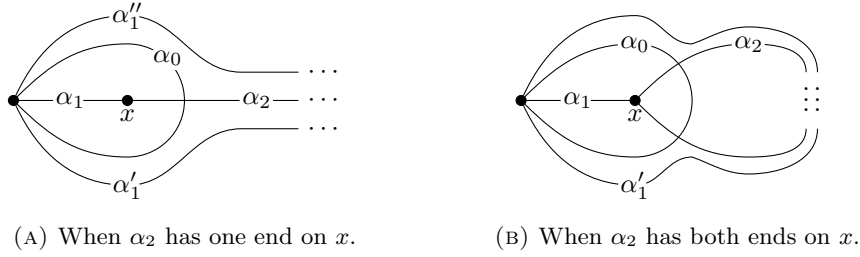
\begin{figure}[ht]
	\centering
	\begin{subfigure}[b]{0.5\textwidth}
		\centering
		\begin{tikzpicture}[scale=0.75, rotate=0]

\node [dot] (a) at (0,0) {};
\node [dot] (b) at (2,0) {};
\coordinate (l) at (2,1);
\coordinate (r) at (2,-1);
\coordinate (t) at (3,0);
\coordinate (t2) at (5,0);

\node [below] at (b) {$x$};
\draw (a) to [out=45,in=180] (l) to [out=0,in=90] node [yshift=1] {\contour*{white}{$\alpha_0$}} (t) to [out=270,in=0] (r) to [out=180,in=-45] (a);
\draw (b) -- node [near end] {\contour*{white}{$\alpha_2$}} (t2) node [right] {$\cdots$};
\draw (a) -- node {\contour*{white}{$\alpha_1$}} (b);
\draw (a) to [out=-65,in=180] ($(r)+(0,-0.5)$) node {\contour*{white}{$\alpha'_1$}} to [out=0,in=180] ($(t)+(1.0,-0.5)$) to [out=0,in=180] ($(t2)+(0,-0.5)$) node [right] {$\cdots$};
\draw (a) to [out=65,in=180] ($(l)+(0,0.5)$) node {\contour*{white}{$\alpha''_1$}} to [out=0,in=180] ($(t)+(1.0,0.5)$) to [out=0,in=180] ($(t2)+(0,0.5)$) node [right] {$\cdots$};

\end{tikzpicture}
		\caption{When $\alpha_2$ has one end on $x$.}
		\label{fig:geodesic_start_one}
	\end{subfigure}%
	~
	\begin{subfigure}[b]{0.5\textwidth}
		\centering
		\begin{tikzpicture}[scale=0.75, rotate=0]

\node [dot] (a) at (0,0) {};
\node [dot] (b) at (2,0) {};
\coordinate (l) at (2,1);
\coordinate (r) at (2,-1);
\coordinate (t) at (3,0);
\coordinate (t2) at (5,0);

\coordinate (l2) at (4,1);
\coordinate (r2) at (4,-1);
\coordinate (t3a) at (5,0.5);
\coordinate (t3b) at (5,-0.5);

\node [below] at (b) {$x$};
\draw (a) to [out=45,in=180] (l) node {\contour*{white}{$\alpha_0$}} to [out=0,in=90] (t) to [out=270,in=0] (r) to [out=180,in=-45] (a);
\draw (b) to [out=45,in=180] (l2) node {\contour*{white}{$\alpha_2$}} to [out=0,in=90] (t3a);
\node at (5,0) {\contour*{white}{$\rvdots$}};
\node at (5.2,0) {\contour*{white}{$\rvdots$}};
\draw (b) to [out=-45,in=180] (r2) to [out=0,in=270] (t3b);
\draw (a) -- node {\contour*{white}{$\alpha_1$}} (b);
\draw (a) to [out=-65,in=180] ($(r)+(0,-0.5)$) node {\contour*{white}{$\alpha'_1$}} to [out=0,in=180] ($(t)+(0,-1)$) to [out=-20,in=180] ($(r2)+(0,-0.3)$) to [out=0,in=270] ($(t3b)+(0.2,0)$);
\draw (a) to [out=65,in=180] ($(l)+(0,0.5)$) to [out=0,in=180] ($(t)+(0,1)$) to [out=20,in=180] ($(l2)+(0,0.3)$) to [out=0,in=90] ($(t3a)+(0.2,0)$);

\end{tikzpicture}
		\caption{When $\alpha_2$ has both ends on $x$.}
		\label{fig:geodesic_start_two}
	\end{subfigure}
	\caption{A detour to avoid a once-marked monogon.}
\end{figure}

By Proposition~\ref{prop:equivalent_crossing}, we have $d(H_\alpha, H_\beta) \geq d(\alpha, \beta)$.
On the other hand, Lemma~\ref{lem:geodesic} shows that if $d(\alpha, \beta) \geq 3$ then $d(H_\alpha, H_\beta) \leq d(\alpha, \beta)$.
Thus, on the large scale, the natural map $\jmath: \Cr(\pol(S)) \hookrightarrow \calA(S)$ preserves distances.
Furthermore, from Proposition~\ref{prop:equivalent_crossing}, if $d(\alpha, \beta) = 1$ then $d(H_\alpha, H_\beta) \leq 2$.
The final possibility is the following:

\begin{lemma}
If $d(\alpha, \beta) = 2$ then $d(H_\alpha, H_\beta) \leq 4$.
\end{lemma}

\begin{proof}
Suppose that $\alpha, \gamma, \beta$ is a geodesic.
There are now three cases to consider.
First, if this geodesic is without folds then $H_\alpha, H_\gamma, H_\beta$ is a geodesic in $\Cr(\pol(S))$ and so $d(H_\alpha, H_\beta) = 2$.
Second, if $\{\alpha, \gamma\}$ and $\{\gamma, \beta\}$ are both folded triangles then necessarily $\alpha$ and $\beta$ both cut off once-marked monogons.
Since $F(S) \geq 3$, there are disjoint arcs $\gamma'$ and $\gamma''$ such that 
\[ \intersection(\gamma', \alpha) = \intersection(\gamma', \gamma) = 0 = \intersection(\gamma'', \gamma) = \intersection(\gamma'', \beta) \]
and the path $\alpha, \gamma', \gamma, \gamma'', \beta$ in $\calA(S)$ is without folds.
Hence by Proposition~\ref{prop:equivalent_crossing} it pulls back to a path $H_\alpha, H_{\gamma'}, H_{\gamma}, H_{\gamma''}, H_\beta$ in $\Cr(P(S))$ and so $d(H_\alpha, H_\beta) \leq 4$.
Third, if the geodesic contains a unique folded triangle then without loss of generality it is formed by $\{\alpha, \gamma\}$.
In this case $\alpha$ must cut off a once-marked monogon (otherwise $\alpha$ and $\beta$ would be disjoint).
Hence again there is an arc $\gamma'$, disjoint from both $\alpha$ and $\gamma$, such that $\alpha, \gamma', \gamma, \beta$ is a path without folds in $\calA(S)$.
Again this pulls back to a path $H_\alpha, H_{\gamma'}, H_{\gamma}, H_\beta$ in $\Cr(P(S))$ and so $d(H_\alpha, H_\beta) \leq 3$.
\end{proof}


\begin{corollary}
\label{cor:qi}
The map $\jmath \from \Cr(\pol(S)) \to \calA(S)$ is a $(1, 2)$--quasi-isometry. \qed
\end{corollary}

Hence $\Cr(\pol(S))$ has the same large scale geometry as $\calA(S)$.
For example, when $\partial S = \emptyset$, Masur--Schleimer showed that $\calA(S)$ is hyperbolic \cite[Theorem~20.2]{MasurSchleimer}.
Hensel--Przytycki--Webb later showed that in this case in fact every geodesic triangle in $\calA(S)$ has a $7$--centre \cite[Theorem~1.2]{HenselPrzytyckiWebb} and so the same is true for $\Cr(\pol(S))$.

For ease of notation, for a hyperplane $H \in \Hyp(\pol(S))$ let $\lnk(H)$ denote the subgraph of $\Cr(\pol(S))$ induced by
\[ \{ H' \in \Hyp(\pol(S)) \; : \; H \crosses H' \} \]

We will now use the crossing graph to prove that, generically, different surfaces have different polygonalisation complexes.
\begin{lemma}
\label{lem:characterise_folded}
Arcs $\alpha$ and $\beta$ form a folded triangle, as shown in Figure~\ref{fig:folded}, if and only if $\lnk(H_\alpha) \subsetneq \lnk(H_\beta)$.
\end{lemma}

\begin{proof}
We use Proposition~\ref{prop:equivalent_crossing} repeatedly to determine whether one hyperplane lies in the link of another.

Suppose $\alpha$ and $\beta$ form a folded triangle.
If $H_\gamma \in \lnk(H_\alpha)$ then $\alpha$ and $\gamma$ are distinct, disjoint and do not form a folded triangle.
Thus $\gamma$ and $\beta$ must also be disjoint and distict.
Additionally, since $\alpha$ only forms a folded triangle with $\beta$ we have that $\gamma$ and $\beta$ do not form a folded triangle and so $H_\gamma \in \lnk(H_\beta)$.
On the other hand, let $\calT \ni \beta$ be a triangulation without folded triangles.
Then there is an arc $\gamma \in \calT$ such that $\intersection(\gamma, \beta) = 0$ but $\intersection(\gamma, \alpha) \neq 0$.
Thus $H_\gamma \in \lnk(H_\beta)$ but $H_\gamma \notin \lnk(H_\alpha)$ and so $\lnk(H_\alpha) \subsetneq \lnk(H_\beta)$.

Suppose $\alpha$ does not form a folded triangle with $\beta$.
Let $\calT \ni \alpha$ be a triangulation with one or zero folded triangles depending on whether $\alpha$ does or does not cut off a once-marked monogon respectively.
There are now three cases to consider:
\begin{enumerate}
\item If $\intersection(\beta, \calT) = 0$ then $\beta \in \calT$.
	Therefore, since $\beta$ does not form a folded triangle with $\alpha$, we have that $H_\beta \in \lnk(H_\alpha)$ but $H_\beta \notin \lnk(H_\beta)$.
\item If $\intersection(\beta, \calT) \neq 0$ and there is an arc $\gamma \in \calT$ such that $\gamma$ and $\alpha$ are distinct and do not form a folded triangle and $\intersection(\gamma, \beta) \neq 0$ then $H_\gamma \in \lnk(H_\alpha)$ but $H_\gamma \notin \lnk(H_\beta)$.
\item If $\intersection(\beta, \calT) \neq 0$ and any arc of $\calT$ that intersects $\beta$ is either $\alpha$ or forms a folded triangle with $\alpha$ then $\alpha$ and $\beta$ are as shown in Figure~\ref{fig:simple_beta}.
	Hence there is an arc $\gamma$ which appears as exactly one side of the unique square of $\calT - \{\alpha\}$.
	Flipping $\gamma$ in $\calT$ gives an arc $\gamma' \neq \alpha$.
	Since $\gamma'$ does not cut off a once-marked monogon, is disjoint from $\alpha$ and $\intersection(\gamma', \beta) \neq 0$ we have that $H_{\gamma'} \in \lnk(H_\alpha)$ but $H_{\gamma'} \notin \lnk(H_\beta)$.
\end{enumerate}
In any case, the inclusion $\lnk(H_\alpha) \subsetneq \lnk(H_\beta)$ does not hold.
\end{proof}

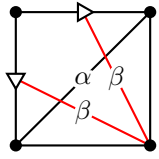
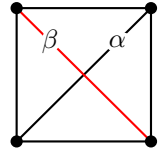
\begin{figure}[ht]
	\centering
	\begin{subfigure}[b]{0.5\textwidth}
		\centering
		\begin{tikzpicture}[scale=1.25,thick]

\foreach \i in {45,135,...,315} {\draw (\i:1) -- (\i+90:1);}
\draw (45:1) -- node {\contour*{white}{$\alpha$}} (225:1);
\draw [red] (0,0.707) -- node [pos=0.5, black] {\contour*{white}{$\beta$}} (315:1);
\draw [red] (-0.707,0) -- node [pos=0.5, black] {\contour*{white}{$\beta$}} (315:1);

\node [tri,rotate=30,scale=2] at (0,0.707) {};
\node [tri,rotate=-60,scale=2] at (-0.707,0) {};
\foreach \i in {45,135,...,315} {\node [dot] at (\i:1) {};}

\end{tikzpicture}
		\caption{When $\beta$ meets a folded triangle.}
	\end{subfigure}%
	~
	\begin{subfigure}[b]{0.5\textwidth}
		\centering
		\begin{tikzpicture}[scale=1.25,thick]

\foreach \i in {45,135,...,315} {\draw (\i:1) -- (\i+90:1);}
\draw (45:1) -- node [near start] {\contour*{white}{$\alpha$}} (225:1);
\draw [red] (135:1) -- node [near start, black] {\contour*{white}{$\beta$}} (315:1);
\foreach \i in {45,135,...,315} {\node [dot] at (\i:1) {};}

\end{tikzpicture}
		\caption{When $\beta$ only meets $\alpha$.}
	\end{subfigure}
	\caption{When $\beta$ only meets $\alpha$ and arcs that form a folded triangle with $\alpha$.}
	\label{fig:simple_beta}
\end{figure}

\begin{theorem}
\label{thrm:pol_isom}
The cube complexes $\pol(S)$ and $\pol(S')$ are isomorphic if and only if $S$ and $S'$ are homeomorphic.
\end{theorem}

\begin{proof}
If $\pol(S)$ and $\pol(S')$ are isomorphic as cube complexes then $\Cr(\pol(S))$ and $\Cr(\pol(S'))$ are isomorphic graphs.
By Proposition~\ref{prop:equivalent_crossing} and Lemma~\ref{lem:characterise_folded}, adding the edges
\[ \{ \{H, H'\} \; : \; H,H' \in \Cr(\pol(S)), \; \lnk(H) \subsetneq \lnk(H') \} \]
to $\Cr(\pol(S))$ produces a graph isomorphic to $\calA(S)$.
Similarly, by adding the edges
\[ \{ \{H, H'\} \; : \; H,H' \in \Cr(\pol(S')), \; \lnk(H) \subsetneq \lnk(H') \} \]
to $\Cr(\pol(S'))$ we obtain a graph isomorphic to $\calA(S')$.
This rule for adding edges is purely combinatorial, that is, it depends only on the graph structure.
Since $\Cr(\pol(S))$ and $\Cr(\pol(S'))$ are isomorphic, we deduce $\calA(S) \isom \calA(S')$.
By rigidity of the arc graph (Theorem~\ref{thrm:arc_isom}), this in turn implies that $S$ is homeomorphic to $S'$.

The reverse direction is straightforward.
\end{proof}

In fact by examining the possible cases shown in Appendix~\ref{sec:exceptions}, we see that this theorem also holds when $F(S) = 2$.

\section{Failure of Gromov's link condition}
\label{sec:rigidity}

In general, $\pol(S)$ contains many vertices that fail Gromov's link condition \cite[Section~4.2.C]{Gromov}, which we rephrase as follows.
When $\partial S = \emptyset$ this is the only reason why $\pol(S)$ fails to be CAT(0), since in this case $\pol(S)$ is contractible.

\begin{definition}
\label{def:pos_curvature}
A \emph{positive curvature system} (based at $P \in \pol(S)$) is a set of $k \geq 3$ edges $e_1, \ldots, e_k$ in $\pol(S)$ incident to $P$ such that:
\begin{itemize}
\item every subset of $\{e_1, \ldots , e_{k} \}$ of size $k - 1$ is contained in an embedded $(k-1)$--cube, and
\item the set $\{e_1, \ldots, e_k \}$ is not contained in an embedded $k$--cube.
\end{itemize}
See Figure~\ref{fig:pos_curvature}.
\end{definition}

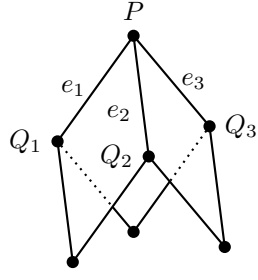
\begin{figure}[ht]
\centering
\begin{tikzpicture}[scale=2, thick]

\node [dot, label=above:{$P$}] (abc) at (0,0) {};
\node [dot, label=left:{$Q_2$}] (bc) at (0.1,-0.8) {};
\node [dot, label=left:{$Q_1$}] (ac) at (-0.5,-0.7) {};
\node [dot, label=right:{$Q_3$}] (ab) at (0.5,-0.6) {};
\node [dot] (a) at ($(ac)+(ab)$) {};
\node [dot] (b) at ($(bc)+(ab)$) {};
\node [dot] (c) at ($(bc)+(ac)$) {};

\draw (abc) -- node [below left] {$e_2$} (bc);
\draw (abc) -- node [left] {$e_1$} (ac);
\draw (abc) -- node [right] {$e_3$} (ab);
\draw (bc) -- (c) -- (ac);
\draw (ab) -- (b) -- (bc);

\coordinate (x) at (intersection of a--ac and c--bc);
\coordinate (y) at (intersection of a--ab and b--bc);
\draw [dotted] (ac) -- (x);
\draw (x) -- (a);
\draw [dotted] (ab) -- (y);
\draw (y) -- (a);

\end{tikzpicture}
\caption{A positive curvature system based at $P$.}
\label{fig:pos_curvature}
\end{figure}

Positive curvature systems are useful as they allow us to determine the direction of inclusion of the edges involved.

\begin{proposition}[Positive curvature criterion]
\label{prop:pos_curvature}
If $e_1, \ldots, e_k$ is a positive curvature system where $e_i = \{P, Q_i\}$ then $P \succ Q_i$.
\end{proposition}

\begin{proof}
Let $\alpha_i = \arc(e_i)$.
For a contradiction, suppose that $P \prec Q_1$.
Then $\alpha_1$ is disjoint from all arcs of $P$.
By assumption, $\{e_2 , \ldots , e_k\}$ lie in an embedded $(k-1)$--cube $C$ that contains $P$.
Recall that $C$ has a unique source $C^+$ and a unique sink $C^-$ with respect to inclusion and that $C^+ - C^- = \{\alpha_2, \ldots, \alpha_k\}$.
Observe that for all $i = 2, \ldots, k$, the edges $e_1$ and $e_i$ are contained in a common $(k-1)$--cube since $\{e_1, \ldots, e_k\}$ is a positive curvature system.
This means that $\alpha_1$ is disjoint from $\alpha_i$ for all $i$.
Note that $C^- \subseteq P \subseteq C^+$ and $C^+ - P \subseteq \{\alpha_2, \ldots, \alpha_k\}$.
We deduce that $\alpha_1$ is disjoint from $C^+$, and so $C^+ \cup \{\alpha_1\}$ is also a polygonalisation.
The set $[C^- , C^+ \cup \{\alpha_1\}]$ contains $P \cup \{\alpha_1\}$, and is therefore the vertex set of a $k$--cube that contains the edges $\{e_1, \ldots, e_k\}$.
This contradicts the assumption that $e_1, \ldots, e_k$ is a positive curvature system.
\end{proof}

\begin{proposition}
\label{lem:characterisation}
Let $e_1, \ldots, e_k$ be edges incident to $P$ with $k \geq 3$.
Let $\alpha_i = \arc(e_i)$.
Then $\{e_1, \ldots, e_k\}$ is a positive curvature system based at $P \in \pol(S)$ if and only if each $\alpha_i \in P$ and there is a (possibly peripheral) simple closed curve $c$ on $S$ such that:
\begin{itemize}
\item $\intersection(c, \alpha_i)= 1$ for every $i$;
\item $\intersection(c, \beta)=0$ for every $\beta \in P - \{ \alpha_1, \ldots, \alpha_k \}$;
\item $c$ meets each polygon of $P$ at most once.
\end{itemize}
\end{proposition}

\begin{proof}
By assumption $Q \defeq P - \{\alpha_2, \ldots, \alpha_k\}$ is a polygonalisation, but $P - \{\alpha_1, \ldots, \alpha_k\}$ is not.
Since $\alpha_1$ is not removable from $Q$, its interior meets a single polygon $\calR$ of $Q$.
Let $c$ be a simple arc in $\calR$ connecting one side of $\alpha_1$ to the other.
By construction $c$ is a simple closed curve in $S$ disjoint from each arc in $P - \{\alpha_1, \ldots, \alpha_k\}$.
Furthermore $c$ intersects each $\alpha_i$ at most once.
We claim that $c$ must intersect every $\alpha_i$ exactly once.
For a contradiction, without loss of generality, suppose $\alpha_k$ is disjoint from $c$. Then by assumption $P - \{\alpha_1, \ldots, \alpha_{k-1}\}$ is a polygonalisation.
So $c$ lies in some polygon and so is null-homotopic, a contradiction.

Conversely, suppose $c$ is a curve satisfying the given conditions.
Orient $c$ and assume, without loss of generality, that the arcs it meets appear in the order $\alpha_1, \ldots, \alpha_k$ (up to cyclic permutation).
Let $\calR_i$ be the polygon of $P$ which contains the segment of $c$ appearing between $\alpha_{i-1}$ and $\alpha_{i}$.
The $\calR_i$'s are well defined since $c$ intersects each $\alpha_i$ exactly once, and are distinct since $c$ meets each polygon of $P$ at most once.
Note that $\alpha_2$ is removable from $P$ since its interior meets distinct regions $\calR_1$ and $\calR_2$.
Proceeding inductively, we can remove $\alpha_i$ from $P - \{ \alpha_2, \ldots, \alpha_{i-1} \}$ for $2 < i \leq k$.
Each removal decreases the number of polygons meeting $c$ by one.
Therefore, the interior of $\alpha_1$ meets only one polygon of $P - \{\alpha_2, \ldots, \alpha_k\}$, and so $P - \{ \alpha_1, \ldots, \alpha_k \}$ is not a polygonalisation.
Hence $\{ e_1, \ldots, e_k \}$ is a positive curvature system.
\end{proof}

Propositions~\ref{prop:pos_curvature} and \ref{lem:characterisation} show that $P \in \pol(S)$ fails Gromov's link condition if and only if its dual fat graph contains an embedded cycle of length at least three.
Hence, positive curvature systems are abundant throughout $\pol(S)$.
In fact, together with the square lemma and one additional piece of information, they can be used to determine the direction of inclusion of all edges.

\begin{proposition}
\label{prop:characterise_containment}
There is a combinatorial criterion such that for each edge $e = \{P, Q\}$ in $\pol(S)$ we have that $P \succ Q$ if and only if the criterion is satisfied by the $(E(S)+2)$--neighbourhood of $e$.
\end{proposition}

\begin{proof}
Suppose that $P \succ Q$ and let $\alpha \defeq \arc(e)$.
We may assume that $e$ is not part of a positive curvature system as otherwise we recover that $P \succ Q$ automatically by Proposition~\ref{prop:pos_curvature}.

Since $\alpha$ is removable from $P$ it meets two distinct polygons $\calR_1$ and $\calR_2$ of $P$.
We consider the different possibilities for how these polygons meet:
\begin{itemize}
	\item Suppose that $\calR_1$ and $\calR_2$ share an additional arc $\beta \in P$.
		Let $c$ be a simple closed curve contained in $\calR_1 \cup \calR_2$ that meets only $\alpha$ and $\beta$.
	\begin{itemize}
		\item Suppose that $\calR_1$ (respectively $\calR_2$) has at least four sides.
			Let $\gamma$ be a diagonal of $\calR_1$ (resp. $\calR_2$) that meets $c$.
			Then the edge of $\pol(S)$ from $P \cup \{\gamma\}$ to $Q \cup \{\gamma\}$ is part of a positive curvature system.
			Therefore using Proposition~\ref{prop:pos_curvature} we have that $P \cup \{\gamma\} \succ Q \cup \{\gamma\}$ and so $P \succ Q$ by Lemma~\ref{lem:square_lemma}.
		\item Otherwise, both $\calR_1$ and $\calR_2$ are triangles.
			Therefore there is a side $\gamma$ ($\neq \alpha, \beta$) of $\calR_1$ or $\calR_2$ that is removable and doing so creates a polygon with at least four sides.
			By the preceding argument, $P - \{\gamma\} \succ Q - \{\gamma\}$ can be deduced from the positive curvature criterion and the square lemma.
			Hence, $P \succ Q$ can too.
	\end{itemize}
	\item Otherwise, $\calR_1$ and $\calR_2$ meet only along $\alpha$.
		In this case, any arc that can be added (respectively removed) from $P$ can also be added (resp.\ removed) from $Q$.
		However there are also arcs that are disjoint from $Q$ but intersect $\alpha$.
		Hence $\deg(P) < \deg(Q)$.
\end{itemize}
Thus we can deduce that $P \succ Q$ either from: the positive curvature criterion applied to a parallel cube at most distance two away, or (if the positive curvature criterion is not definitive) the degrees of $P$ and $Q$ in $\pol(S)$.

Furthermore, any cube in $\pol(S)$ has diameter at most $E(S)$.
Hence $P \succ Q$ can be determined from the combinatorics of the $(E(S) + 2)$--neighbourhood of $e$.
\end{proof}

\begin{corollary}
\label{cor:deficiency_combinatorics}
For each $k$, there is a combinatorial criterion that characterises the vertices of $\pol(S)$ corresponding to the polygonalisations with $k$ arcs.
\end{corollary}

\begin{proof}
This follows from the fact that a vertex $P$ corresponds to a polygonalisation with $k$ arcs if and only if
\[ \exists P_1 \prec \ldots \prec P_{E(S) - k} \; (P \prec P_1 \wedge \nexists \; Q \; (P_{E(S) - k} \prec Q)). \]
Since $\prec$ can be expressed as a combinatorial criterion by Proposition~\ref{prop:characterise_containment}, this statement is a combinatorial criterion also.
\end{proof}

We will now use these criteria to prove the rigidity of $\pol(S)$.
\begin{lemma}
\label{lem:pol_restr}
Every automorphism of $\pol(S)$ induces an automorphism of $\flip(S)$ by restriction of the vertices that correspond to triangulations. 
\end{lemma}

\begin{proof}
Suppose that $\phi \in \Aut(\pol(S))$ is an automorphism. 
By Corollary \ref{cor:deficiency_combinatorics}, 
the map $\phi$ restricts to an automorphism of the graph  
\[ \mathcal{F}' \defeq \{P \in \pol(S) : |P| \geq E(S) - 1\} \]
induced by the set of triangulations and triangulations with one arc missing (recall that $|P|$ equals the number of arcs in $P$). 
Moreover $\phi$ maps triangulations to triangulations. 
Since $\mathcal{F}'$ is the barycentric subdivision of $\flip(S)$, it follows that $\phi$ induces an automorphism $\phi' \in \Aut(\flip(S))$. 
\end{proof}

\begin{theorem}
\label{thrm:pol_aut}
The natural homomorphism
\[\rho: \EMod(S) \to \Aut(\pol(S)) \]
is an isomorphism.
\end{theorem}

\begin{proof}
If $f \in \mathrm{Ker}(\rho)$ then it must induce the identity on $\flip(S)$. 
By Theorem~\ref{thrm:flip_rig} the map $f$ must be the trivial mapping class and so $\rho$ is injective.

We now prove that $\rho$ is surjective. 
By Lemma \ref{lem:pol_restr} any $\phi \in \Aut(\pol(S))$ induces some $\phi' \in \Aut (\flip(S))$. 
By Theorem \ref{thrm:flip_rig}, $\phi'$ is induced by some $f \in \EMod(S)$. 
Let $f_\star $ be the automorphism of $\pol(S)$ induced by $f$. 
We will now prove that $f_\star(P)  = \phi(P)$ for every $P \in \pol(S)$.  
Denote by $\mathcal{U}(P)$ the set of all triangulations that contain $P$. 
By Lemma \ref{lem:pol_restr} the map $\phi$ preserves the set of all triangulations. 
By Proposition \ref{prop:characterise_containment} the map $\phi$  preserves the direction of the inclusions. 
Hence, we have $\mathcal{U}(\phi(P)) = \phi(\mathcal{U}(P)) $.
Similarly, $\mathcal{U}(f_\star(P)) = f_\star(\mathcal{U}(P))$. 
Since $f_\star$ and $\phi$ coincide on every triangulation, we have $\mathcal{U}(f_\star(P)) = \mathcal{U}(\phi(P))$. 
Note that $\mathcal{U}(Q) \supseteq \mathcal{U}(Q')$ if and only if $Q \subseteq Q'$.
It follows that $f_\star(P) = \phi(P)$, hence $\rho(f) = \phi$. 
\end{proof}

\begin{acknowledgements}
The authors are grateful for the hospitality of the University of Illinois, Indiana University, University of Oklahoma, ICERM and MSRI.
The authors acknowledge support from U.S. National Science Foundation grants DMS 1107452, 1107263, 1107367 ``RNMS: GEometric structures And Representation varieties'' (the GEAR Network).
The second author acknowledges support from IU Provost's Travel Award for Women in Science.
This material is based upon work supported by the National Science Foundation under Grant No.\ DMS 1440140 while the second and third authors were in residence at the Mathematical Sciences Research Institute in Berkeley, California, during the Fall 2016 semester.
\end{acknowledgements}

\appendix

\section{Examples of polygonalisation complexes}
\label{sec:examples}

Here we use the notation $S_{g,s}^{p_1, \ldots, p_b}$ to denote the surface of genus $g$ with $s$ marked points in its interior and $b$ boundary components with $p_1, \ldots, p_b$ marked points respectively.
In the following examples we colour a polygonalisation $P$ red, blue or green depending on whether $E(S) - |P|$ is zero, one or two respectively.
Hence the subgraph induced by the red and blue vertices is homeomorphic to the flip graph of the surface.
In each example we highlight one top-dimensional cube in each $\EMod(S)$--orbit. In each of these cases $F(S) = 3$ and the only omitted case with $F(S) = 3$ is $S_{1,0}^{1}$.
Additionally, the polygonalisation complex of a hexagon, where $F(S_{0,0}^6) = 4$, is shown in Figure~\ref{fig:P_S_0_0^6}. Many of these complexes were generated using \texttt{fatter} \cite{fatter}.

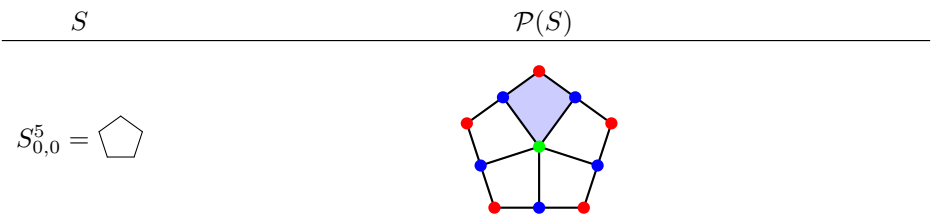
\begin{longtable}{cc}
$S$ & $\pol(S)$ \\ \hline
$S_{0,0}^{5} = $
\tikz[scale=0.3,line cap=round,baseline=-3]{
	\foreach \i in {90,162,...,378} {\draw (\i:1) -- (\i+72:1);}
} & \begin{tikzpicture}[scale=2, thick, baseline=0]

\pgfmathsetmacro{\inner}{0.5}

\fill [blue, opacity=0.2] (0,0) -- ($(18:\inner)!0.5!(90:\inner)$) -- (90:\inner) -- ($(90:\inner)!0.5!(90+72:\inner)$);

\foreach \i in {90,162,...,378} {
	\draw (\i:\inner) -- (\i+72:\inner);
	\draw ($(\i:\inner)!0.5!(\i+72:\inner)$) -- (0,0);
}
\foreach \i in {90,162,...,378} {
	\node [dot, red] at (\i:\inner) {};
	\node [dot, blue] at ($(\i:\inner)!0.5!(\i+72:\inner)$) {};
}
\node [dot, green] at (0,0) {};

\path [use as bounding box] (0,0.7) -- (0,-0.6);

\end{tikzpicture} \\
$S_{0,1}^{3} = $
\tikz[scale=0.3,line cap=round,baseline=-3]{
	\foreach \i in {90,162,...,378} {\draw (\i:1) -- (\i+72:1);}
	\node [tri,rotate=66] at ($(90:1)!0.5!(90+72:1)$) {};
	\node [tri,rotate=-66] at ($(90:1)!0.5!(90-72:1)$) {};
} & \begin{tikzpicture}[scale=1.1,thick,baseline=0]

\newcommand{\upV}{90:1};
\newcommand{\leftV}{150:1};
\newcommand{\rightV}{30:1};

\coordinate (V) at (\upV);
\coordinate (L) at (\leftV);
\coordinate (R) at (\rightV);

\coordinate (A) at ($(V)+(L)$);
\coordinate (B) at ($(V)+(R)$);
\coordinate (C) at ($(0,0)!1!-120:(A)$);
\coordinate (D) at ($(0,0)!1!-120:(V)$);
\coordinate (X) at ($0.25*(B)+0.25*(C)$);

\fill [blue, opacity=0.2] (0,0) -- ($0.5*(V)$) -- (X) -- ($0.5*(D)$) -- cycle;
\fill [gray, opacity=0.2] ($0.5*(V)$) -- (V) -- ($(V)+0.5*(R)$) -- (X) -- cycle;
\fill [red, opacity=0.2] ($(B)$) -- ($0.5*(B)+0.5*(C)$) -- (X) -- ($(V)+0.5*(R)$) -- cycle;

\foreach \i in {0,120,240} {
\begin{scope}[rotate=\i]
	\coordinate (V) at (\upV);
	\coordinate (L) at (\leftV);
	\coordinate (R) at (\rightV);
	
	\coordinate (A) at ($(V)+(L)$);
	\coordinate (B) at ($(V)+(R)$);
	\coordinate (C) at ($(0,0)!1!-120:(A)$);
	\coordinate (D) at ($(0,0)!1!-120:(V)$);
	\coordinate (X) at ($0.25*(B)+0.25*(C)$);
	\draw (0,0) -- (V);
	\draw (V) -- (A);
	\draw (V) -- (B);
	\draw (B) -- (C);
	\draw (X) -- ($0.5*(V)$);
	\draw (X) -- ($0.5*(D)$);
	\draw (X) -- ($(V)+0.5*(R)$);
	\draw (X) -- ($2*(X)$);
	\draw (X) -- ($0.5*(C)+0.5*(D)$);
\end{scope}
}

\foreach \i in {0,120,240} {
\begin{scope}[rotate=\i]
	\coordinate (V) at (\upV);
	\coordinate (L) at (\leftV);
	\coordinate (R) at (\rightV);
	
	\coordinate (A) at ($(V)+(L)$);
	\coordinate (B) at ($(V)+(R)$);
	\coordinate (C) at ($(0,0)!1!-120:(A)$);
	\coordinate (D) at ($(0,0)!1!-120:(V)$);
	\coordinate (X) at ($0.25*(B)+0.25*(C)$);
	\node [dot, green] at (X) {};
	\node [dot, blue] at ($0.5*(V)$) {};
	\node [dot, red] at (V) {};
	\node [dot, red] at (A) {};
	\node [dot, red] at (B) {};
	\node [dot, blue] at ($(V)+0.5*(L)$) {};
	\node [dot, blue] at ($(V)+0.5*(R)$) {};
	\node [dot, blue] at ($2*(X)$) {};
\end{scope}
}


\node [dot, red] at (0,0) {};

\path [use as bounding box] (0,1.7) -- (0,-1.7);

\end{tikzpicture} \\
$S_{0,0}^{2,1} = $
\tikz[scale=0.3,line cap=round,baseline=-3]{
	\foreach \i in {90,162,...,378} {\draw (\i:1) -- (\i+72:1);}
	\node [tri,rotate=18] at ($(90+72+72:1)!0.5!(90+72:1)$) {};
	\node [tri,rotate=-18] at ($(90-72-72:1)!0.5!(90-72:1)$) {};
} & \begin{tikzpicture}[xscale=1.75, yscale=1.1, thick, baseline=0]

\fill [blue, opacity=0.2] (0.5,-0.75) -- (0,-0.625) -- (0.25,0) -- (0.5,-0.25) -- cycle;
\fill [gray, opacity=0.2] (0,-1.25) -- (0.5,-1.25) -- (0.5,-0.75) -- (0,-0.625) -- cycle;
\fill [red, opacity=0.2] (0,-0.625) -- (0.25,0) -- (0,0.25) -- (-0.25,0) -- cycle;

\draw (-2.25,-1.25) -- (2.25,-1.25);
\draw (-2.25,1.25) -- (2.25,1.25);

\foreach \i / \s in {-1.5/1, -1/-1, -0.5/1, 0/-1, 0.5/1, 1/-1, 1.5/1}{
\begin{scope}[shift={(\i,0)}, yscale=\s, xscale=0.5]
	\draw (-1,1.25) -- (-1,0.25) -- (0,-0.25) -- (1,0.25) -- (1,1.25);
	\foreach \p in {(0,1.25),(-1,0.75),(-0.5,0),(0.5,0),(1,0.75)} {\draw [thin] (0,0.625) -- \p;}
	\foreach \p in {(0,1.25),(-1,0.75),(-0.5,0),(0.5,0),(1,0.75)} {\node [dot,blue] at \p {};}
	\foreach \p in {(-1,1.25),(-1,0.25),(0,-0.25),(1,0.25),(1,1.25)} {\node [dot,red] at \p {};}
	\node [dot,green] at (0,0.625) {};
\end{scope}
}

\node at (-2,0) {$\cdots$};
\node at (2,0) {$\cdots$};

\path [use as bounding box] (0,1.7) -- (0,-1.7);

\end{tikzpicture} \\
$S_{0,2}^{1} = $
\tikz[scale=0.3,line cap=round,baseline=-3]{
	\foreach \i in {90,162,...,378} {\draw (\i:1) -- (\i+72:1);}
	\node [tri,rotate=66] at ($(90:1)!0.5!(90+72:1)$) {};
	\node [tri,rotate=-66] at ($(90:1)!0.5!(90-72:1)$) {};
	\node [tri,rotate=18,fill=black] at ($(90+72+72:1)!0.5!(90+72:1)$) {};
	\node [tri,rotate=-18,fill=black] at ($(90-72-72:1)!0.5!(90-72:1)$) {};
} & \begin{tikzpicture}[xscale=1.2, yscale=1.25,thick,baseline=0]

\coordinate (leftV) at (-0.4,-1);
\coordinate (rightV) at (0.4,-0.75);

\node at (3.75,0) {$\cdots$};
\node at (-3.75,0) {$\cdots$};

\fill [blue, opacity=0.2] (-0.5,1.5) -- (-0.25,1.25) -- (-0.5,0.75) -- (-0.75,1.25) -- cycle;
\fill [gray, opacity=0.2] (-0.25,1.25) -- (0,1) -- (0,0.5) -- (-0.5,0.75) -- cycle;
\fill [red, opacity=0.2] (-0.5,0.75) -- (0,0.5) -- (0,0) -- (-0.5,0) -- cycle;
\fill [green, opacity=0.2] (-0.5,0) -- (0,0) -- (0.5,0) -- ($(0,0)+0.5*(leftV)$) -- cycle;
\fill [black, opacity=0.2] (0.5,0) -- (1,0) -- ($(1,0)+0.5*(leftV)$) -- ($(0,0)+0.5*(leftV)$) -- cycle;
\fill [orange, opacity=0.2] ($(0,0)+0.5*(leftV)$) -- ($(1,0)+0.5*(leftV)$) -- ($(1,0)+(leftV)$) -- ($(0,0)+(leftV)$) -- cycle;

\draw (-3.5,0) -- (3.5,0);
\foreach \i in {-3,-2,...,3} {\draw (\i,0) -- (\i,1);}
\foreach \i in {-3,-2,...,2} {\draw (\i,1) -- (\i+0.5,1.5); \draw (\i+0.5,1.5) -- (\i+1,1);}

\foreach \i in {-3,-1,1} {\draw (\i,0) -- ($(\i,0)+(leftV)$) -- ($(\i+2,0)+(leftV)$) -- (\i+2,0);}
\draw [dotted] (-2,0) -- ($(-2,0)+(rightV)$) -- ($(2,0)+(rightV)$) -- (2,0);
\draw [dotted] (0,0) -- (rightV);
\foreach \i in {-1,1} {
	\draw [dotted, thin] ($(\i,0)+0.5*(rightV)$) -- (\i-0.5,0);
	\draw [dotted, thin] ($(\i,0)+0.5*(rightV)$) -- (\i+0.5,0);
	\draw [dotted, thin] ($(\i,0)+0.5*(rightV)$) -- ($(\i,0)+(rightV)$);
	\draw [dotted, thin] ($(\i,0)+0.5*(rightV)$) -- ($(\i-1,0)+0.5*(rightV)$);
	\draw [dotted, thin] ($(\i,0)+0.5*(rightV)$) -- ($(\i+1,0)+0.5*(rightV)$);
	}
\foreach \i in {-2,0,2} {\node [dot, red!40] at ($(\i,0)+(rightV)$) {};}
\foreach \i in {-2,0,2} {\node [dot, blue!40] at ($(\i,0)+0.5*(rightV)$) {};}
\foreach \i in {-1,1} {\node [dot, blue!40] at ($(\i,0)+(rightV)$) {};}
\foreach \i in {-1,1} {\node [dot, green!40] at ($(\i,0)+0.5*(rightV)$) {};}

\foreach \i in {-3,-2,...,2} {
	\draw [thin] (\i+0.5,0.75) -- (\i+0.5,0);
	\draw [thin] (\i+0.5,0.75) -- (\i,0.5);
	\draw [thin] (\i+0.5,0.75) -- (\i+1,0.5);
	\draw [thin] (\i+0.5,0.75) -- (\i+0.25,1.25);
	\draw [thin] (\i+0.5,0.75) -- (\i+0.75,1.25);
	}

\foreach \i in {-3,-1,1} {
	\draw [thin] ($(\i+1,0)+0.5*(leftV)$) -- (\i+0.5,0);
	\draw [thin] ($(\i+1,0)+0.5*(leftV)$) -- (\i+1.5,0);
	\draw [thin] ($(\i+1,0)+0.5*(leftV)$) -- ($(\i,0)+0.5*(leftV)$);
	\draw [thin] ($(\i+1,0)+0.5*(leftV)$) -- ($(\i+2,0)+0.5*(leftV)$);
	\draw [thin] ($(\i+1,0)+0.5*(leftV)$) -- ($(\i+1,0)+(leftV)$);
	}

\foreach \i in {-3,-2,...,2} {
	\node [dot, green] at (\i+0.5,0.75) {};
	
	\node [dot, blue] at (\i+0.5,0) {};
	\node [dot, blue] at (\i,0.5) {};
	\node [dot, blue] at (\i+0.25,1.25) {};
	\node [dot, blue] at (\i+0.75,1.25) {};
	
	\node [dot, red] at (\i,0) {};
	\node [dot, red] at (\i,1) {};
	\node [dot, red] at (\i+0.5,1.5) {};
	}
\node [dot, red] at (3,0) {};
\node [dot, blue] at (3,0.5) {};
\node [dot, red] at (3,1) {};

\foreach \i in {-3,-1,1} {\node [dot, green] at ($(\i+1,0)+0.5*(leftV)$) {};}
\foreach \i in {-3,-1,1,3} {
	\node [dot, blue] at ($(\i,0)+0.5*(leftV)$) {};
	\node [dot, red] at ($(\i,0)+(leftV)$) {};
	}
\foreach \i in {-2,0,2} {\node [dot, blue] at ($(\i,0)+(leftV)$) {};}

\path [use as bounding box] (0,1.7) -- (0,-1.7);

\end{tikzpicture} \\
\caption{Examples of polygonalisation complexes.}
\label{tab:examples}
\end{longtable}

\section{Examples of hyperplanes}

Some examples of hyperplanes in polygonalisation complexes. There are examples of CAT(0)/non CAT(0) hyperplanes and examples in which the hyperplanes are not convex subsets.

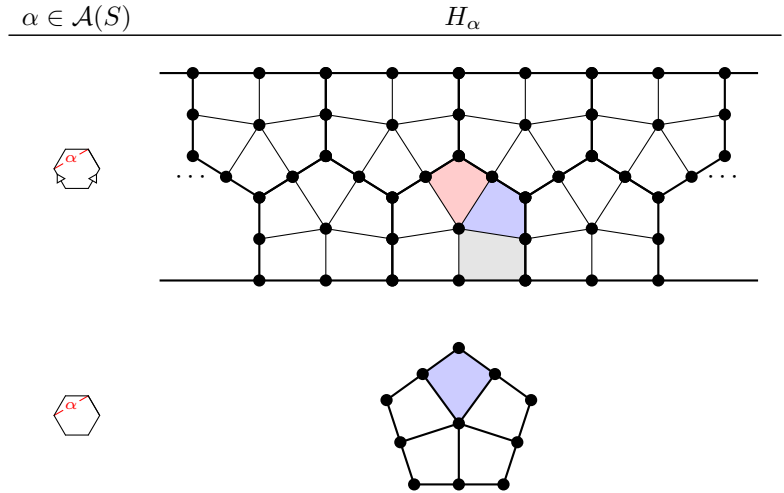
\begin{longtable}{cc}
$\alpha \in \calA(S)$ & $H_\alpha$ \\ \hline
\tikz[scale=0.3,line cap=round,baseline=-3]{
	\draw [red] (60:1) -- node {\contour*{white}{\tiny $\alpha$}} (180:1);
	\foreach \i in {0,60,...,360} {\draw (\i:1) -- (\i+60:1);}
	\node [tri,rotate=-30] at ($(0:1)!0.5!(-60:1)$) {};
	\node [tri,rotate=30] at ($(180:1)!0.5!(240:1)$) {};
} & \begin{tikzpicture}[xscale=1.75, yscale=1.1, thick, baseline=0]

\fill [blue, opacity=0.2] (0.5,-0.75) -- (0,-0.625) -- (0.25,0) -- (0.5,-0.25) -- cycle;
\fill [gray, opacity=0.2] (0,-1.25) -- (0.5,-1.25) -- (0.5,-0.75) -- (0,-0.625) -- cycle;
\fill [red, opacity=0.2] (0,-0.625) -- (0.25,0) -- (0,0.25) -- (-0.25,0) -- cycle;

\draw (-2.25,-1.25) -- (2.25,-1.25);
\draw (-2.25,1.25) -- (2.25,1.25);

\foreach \i / \s in {-1.5/1, -1/-1, -0.5/1, 0/-1, 0.5/1, 1/-1, 1.5/1}{
\begin{scope}[shift={(\i,0)}, yscale=\s, xscale=0.5]
	\draw (-1,1.25) -- (-1,0.25) -- (0,-0.25) -- (1,0.25) -- (1,1.25);
	\foreach \p in {(0,1.25),(-1,0.75),(-0.5,0),(0.5,0),(1,0.75)} {\draw [thin] (0,0.625) -- \p;}
	\foreach \p in {(0,1.25),(-1,0.75),(-0.5,0),(0.5,0),(1,0.75)} {\node [dot] at \p {};}
	\foreach \p in {(-1,1.25),(-1,0.25),(0,-0.25),(1,0.25),(1,1.25)} {\node [dot] at \p {};}
	\node [dot] at (0,0.625) {};
\end{scope}
}

\node at (-2,0) {$\cdots$};
\node at (2,0) {$\cdots$};

\path [use as bounding box] (0,1.7) -- (0,-1.7);

\end{tikzpicture} \\
\tikz[scale=0.3,line cap=round,baseline=-3]{
	\draw [red] (60:1) -- node {\contour*{white}{\tiny $\alpha$}} (180:1);
	\foreach \i in {0,60,...,360} {\draw (\i:1) -- (\i+60:1);}
} & \begin{tikzpicture}[scale=2, thick, baseline=0]

\pgfmathsetmacro{\inner}{0.5}

\fill [blue, opacity=0.2] (0,0) -- ($(18:\inner)!0.5!(90:\inner)$) -- (90:\inner) -- ($(90:\inner)!0.5!(90+72:\inner)$);

\foreach \i in {90,162,...,378} {
	\draw (\i:\inner) -- (\i+72:\inner);
	\draw ($(\i:\inner)!0.5!(\i+72:\inner)$) -- (0,0);
}
\foreach \i in {90,162,...,378} {
	\node [dot] at (\i:\inner) {};
	\node [dot] at ($(\i:\inner)!0.5!(\i+72:\inner)$) {};
}
\node [dot] at (0,0) {};

\path [use as bounding box] (0,0.7) -- (0,-0.6);

\end{tikzpicture} \\
\tikz[scale=0.3,line cap=round,baseline=-3]{
	\draw [red] (60:1) -- node {\contour*{white}{\tiny $\alpha$}} (240:1);
	\foreach \i in {0,60,...,360} {\draw (\i:1) -- (\i+60:1);}
} & \begin{tikzpicture}[scale=2, thick, baseline=0]

\pgfmathsetmacro{\inner}{0.5}

\fill [blue, opacity=0.2] (0,0) -- ($(45:\inner)!0.5!(135:\inner)$) -- (45:\inner) -- ($(45:\inner)!0.5!(315:\inner)$);

\foreach \i in {45,135,225,315} {
	\draw (\i:\inner) -- (\i+90:\inner);
	\draw ($(\i:\inner)!0.5!(\i+90:\inner)$) -- (0,0);
}
\foreach \i in {45,135,225,315} {
	\node [dot] at (\i:\inner) {};
	\node [dot] at ($(\i:\inner)!0.5!(\i+90:\inner)$) {};
}
\node [dot] at (0,0) {};

\path [use as bounding box] (0,0.7) -- (0,-0.6);

\end{tikzpicture} \\
\caption{Examples of hyperplanes in polygonalisation complexes.}
\label{tab:example_hyperplanes}
\end{longtable}

\section{Exceptional surfaces}
\label{sec:exceptions}

In Table~\ref{tab:exceptions} we list the exceptional surfaces, that is, the surfaces with $F(S) < 3$.
Again we use the notation of Appendix~\ref{sec:examples} to describe these surfaces.
This table shows the standing of Theorem~\ref{thrm:pol_isom} and Theorem~\ref{thrm:pol_aut} in these cases.
Examination of $\pol(S)$ in these cases also shows that variants of some of our results, such as Theorem~\ref{thrm:pol_sageev}\footnote{For $S_{1,1}$ the hyperplanes are separating but the two complementary components are not $\pol_\alpha(S)$ and $\overline{\pol_\alpha}(S)$.}, also hold when $F(S) < 3$.

\begin{table}[ht]
\begin{tabular}{ccccc}
$S$ & $\pol(S)$ & $\Aut(\pol(S))$ & $\EMod(S) \to \Aut(\pol(S))$ \\ \hline
$S_{1,1}$ & \begin{tikzpicture}[scale=1, thick, baseline=0]
	\draw (90:0.5) -- (270:0.5);
	\node [dot, blue, scale=1] at (0,0) {};
	
	\foreach \s / \n / \r / \rr in {45 / 2 / 0.5 / 1, 22.5 / 4 / 1 / 1.25, 11.25 / 8 / 1.25 / 1.375, 5.625 / 16 / 1.375 / 1.4375}{
		\foreach \i in {1,2,...,\n}{
			\draw (2*\s+4*\s*\i:\r) -- (2*\s+4*\s*\i-\s:\rr);
			\draw (2*\s+4*\s*\i:\r) -- (2*\s+4*\s*\i+\s:\rr);
			\node [dot, red, scale=2/\n] at (2*\s+4*\s*\i:\r) {};
			\node [dot, blue, scale=2/\n] at ($(2*\s+4*\s*\i:\r)!0.5!(2*\s+4*\s*\i-\s:\rr)$) {};
			\node [dot, blue, scale=2/\n] at ($(2*\s+4*\s*\i:\r)!0.5!(2*\s+4*\s*\i+\s:\rr)$) {};
		}
	}
	\draw [dotted] (0,0) circle [radius=1.5];
	\draw [draw=none] (0,0) circle [radius=1.7]; 
\end{tikzpicture} & Uncountable \cite{KorkmazPapadopoulos} & Homomorphism \\
$S_{0,3}$ & \begin{tikzpicture}[scale=1, thick, baseline=0]

\draw (0,0) -- (90:1);
\draw (0,0) -- (210:1);
\draw (0,0) -- (330:1);

\node [dot, red] at (0,0) {};
\node [dot, red] at (90:1) {};
\node [dot, red] at (210:1) {};
\node [dot, red] at (330:1) {};

\node [dot, blue] at (90:0.5) {};
\node [dot, blue] at (210:0.5) {};
\node [dot, blue] at (330:0.5) {};

\path (0,-0.6) -- (0,1.1); 

\end{tikzpicture} & $\Sym(3)$ & Epimorphism \\
$S_{0,1}^{2}$ & \begin{tikzpicture}[xscale=0.5, yscale=0.25, thick, baseline=0]

\draw (-2,1) -- (-1,-1) -- (0,1) -- (1,-1) -- (2,1);
\node [dot, red] at (-2,1) {};
\node [dot, red] at (0,1) {};
\node [dot, red] at (2,1) {};
\node [dot, blue] at (-1,-1) {};
\node [dot, blue] at (1,-1) {};

\path (-2.1,1.8) -- (2.1,-1.8); 

\end{tikzpicture} & $\ZZ_2$ & Epimorphism \\
$S_{0,0}^{4}$ & \begin{tikzpicture}[xscale=0.5, yscale=0.25, thick, baseline=0]

\draw (-1,1) -- (0,-1) -- (1,1);
\node [dot, red] at (-1,1) {};
\node [dot, blue] at (0,-1) {};
\node [dot, red] at (1,1) {};

\path (-1.1,1.8) -- (1.1,-1.8); 

\end{tikzpicture} & $\ZZ_2$ & Epimorphism \\
$S_{0,0}^{1,1}$ & \begin{tikzpicture}[xscale=1, yscale=0.5, thick, baseline=0]

\draw (-1,1) -- (-1.25,0.5);
\draw (2,1) -- (2.25,0.5);
\foreach \x in {-1,0,1} {
	\draw (\x,1) -- (\x+0.5,0) -- (\x+1,1);
	\node [dot, red] at (\x,1) {};
	\node [dot, blue] at (\x+0.5,0) {};
	}
\node [dot, red] at (2,1) {};
\node at (2.5,0.5) {$\cdots$};
\node at (-1.5,0.5) {$\cdots$};

\path (-1.6,1.4) -- (2.6,-0.4); 

\end{tikzpicture} & $\ZZ_2 \ast \ZZ_2$ & Homomorphism \\
$S_{0,1}^{1}$ & \begin{tikzpicture}[scale=1, thick, baseline=0]

\node [dot, red] {};

\path (-0.5,0.5) -- (0.5,-0.5); 

\end{tikzpicture} & $\one$ & Epimorphism \\
$S_{0,0}^{3}$ & \begin{tikzpicture}[scale=1, thick, baseline=0]

\node [dot, red] {};

\path (-0.5,0.5) -- (0.5,-0.5); 

\end{tikzpicture} & $\one$ & Epimorphism \\
$S_{0,2}$ & $\emptyset$ & $\one$ & Epimorphism \\[1em]
$S_{0,0}^{2}$ & $\emptyset$ & $\one$ & Epimorphism \\[1em]
$S_{0,0}^{1}$ & $\emptyset$ & $\one$ & Epimorphism \\[1em]
$S_{0,1}$ & $\emptyset$ & $\one$ & Epimorphism \\[1em]
\end{tabular}
\caption{The exceptional surfaces.}
\label{tab:exceptions}
\end{table}

\clearpage

\bibliographystyle{amsplain}
\bibliography{bibliography}

\end{document}